\documentclass[12pt]{amsart}
\usepackage{tikz}
\usetikzlibrary{calc, arrows, decorations.markings} \tikzset{>=latex}
\usepackage{amsmath,amssymb,amsfonts}
\usepackage{amscd, amssymb, latexsym, amsmath, amscd}
\usepackage[all]{xy}
\usepackage{pb-diagram}
\usepackage{enumerate}
\usepackage{anysize}
\usepackage[sc]{mathpazo} 
\usepackage{enumitem}
\marginsize{2.5cm}{2.5cm}{2.5cm}{2.5cm}
\usepackage{graphics}
\usepackage{color}
\DeclareFontFamily{U}{wncy}{}
    \DeclareFontShape{U}{wncy}{m}{n}{<->wncyr10}{}
    \DeclareSymbolFont{mcy}{U}{wncy}{m}{n}
    \DeclareMathSymbol{\Sh}{\mathord}{mcy}{"58}

\marginsize{3cm}{3cm}{3cm}{3cm}
\input xy
\xyoption{all}
\usepackage{pb-diagram}
\usepackage[all]{xy}
\input xy
\xyoption{all}
\theoremstyle{plain}
\newtheorem{theorem}{Theorem}[section]

\newtheorem{lemma}{Lemma}

\newtheorem{corollary}{Corollary}

\newtheorem{proposition}{Proposition}

\newtheorem{conj}{Conjecture}

\theoremstyle{definition} \theoremstyle{definition}
\newtheorem{remark}{Remark}

\theoremstyle{remark}

\newcommand{\G}{\textsc{\G}}

\newcommand{\A}{\mathbb{A}}

\newcommand{\PD}{{\rm PD}}

\newcommand{\U}{\mathcal{U}}

\newcommand{\Z}{\mathbb{Z}}

\newcommand{\R}{\mathbb{R}}

\newcommand{\C}{\mathbb{C}}
\newcommand{\Sc}{\mathcal S}
\newcommand{\F}{\mathbb{F}}
\newcommand{\N}{\mathbb{N}}

\newcommand{\Nm}{\mathbb{N}m}

\newcommand{\Hom}{{\rm Hom}}
\newcommand{\SHom}{{\rm SHom}}

\def\G{{\rm G}}

\def\SL{{\rm SL}}

\def\GSp{{\rm GSp}}

\def\Sp{{\rm Sp}}

\def\St{{\rm St}}

\def\SU{{\rm SU}}
\def\U{{\rm U}}

\def\GO{{\rm GO}}
\def\GL{{\rm GL}}
\def\PGL{{\rm PGL}}

\def\Gal{{\rm Gal}}
\def\SO{{\rm SO}}
\def\OO{{\rm O}}

\def\O{{\mathcal O}}

\begin{document}

{
\title{Symplectic  models for  Unitary groups}
\author{Sarah Dijols and Dipendra Prasad}
\address{Aix Marseille Universit\'e, 13453, Marseille, France.  } 
\address{Tata Institute of Fundamental 
Research, Homi Bhabha Road, Bombay - 400 005, India.}
\address{Laboratory of Modern Algebra and Applications, Saint-Petersburg State University, Russia.}

\email{sarah.dijols@univ-amu.fr}
\email{prasad.dipendra@gmail.com}

\maketitle

{\hfill \today}
\begin{abstract}
In analogy with the study of representations of $\GL_{2n}(F)$ distinguished by $\Sp_{2n}(F)$, where $F$ is a 
local field, 
we study representations of $\U_{2n}(F)$ distinguished by $\Sp_{2n}(F)$ in this paper. 
(Only quasi-split unitary groups are considered in this paper since they are the only ones which contain $\Sp_{2n}(F)$.) 
We prove that 
there are no cuspidal representations  of $\U_{2n}(F)$ distinguished by $\Sp_{2n}(F)$ for $F$ a 
non-archimedean local field. We also prove the corresponding 
global theorem that 
there are no cuspidal representations  of $\U_{2n}(\A_k)$ with nonzero period integral on  $\Sp_{2n}(k) \backslash \Sp_{2n}(\A_k)$
for $k$ any number field or a function field.  We  completely classify representations of quasi-split unitary group 
in four variables over local and global fields
with nontrivial symplectic periods using
methods of theta correspondence. We propose a conjectural answer for the classification of  all representations of 
a quasi-split unitary group distinguished by $\Sp_{2n}(F)$. 
\end{abstract}

\begin{tableofcontents}
\end{tableofcontents}
\section{Introduction}

Among the many examples studied about automorphic representations of $G(\A)$ which have nonzero period integrals (where $\A$ is the adele ring of a number field $k$):
$$\int_{H(k)\backslash H(\A)}f(h) dh \not \equiv 0,$$
for $f \in \Pi$, an automorphic representation of $G(\A)$, for $G$ a reductive algebraic group over the number field $k$, and $H$ an algebraic subgroup of $G$ defined over $k$, one of the most complete and beautiful works is due to O. Offen and E. Sayag about symplectic periods of automorphic forms on $\GL_{2n}(\A)$ 
(for $H(\A) = \Sp_{2n}(\A)$), cf. [OS1] and [OS2] for both local and global results for the pair $(\GL_{2n},\Sp_{2n})$.

One of the early results on symplectic periods is due to Heumos and Rallis cf. [HR], who proved that there are no cuspidal representations of $\GL_{2n}(\A)$ with nonzero symplectic period since in fact there are no generic representations of $\GL_{2n}(F)$ which are distinguished by $\Sp_{2n}(F)$, for $F$ a non-archimedean local field. (For a subgroup $H$ of a  group
$G$, a representation $\pi$ of $G$ is said to be distinguished by $H$ if there exists a nonzero linear form $\ell: \pi \rightarrow \C$ such that $\ell(hv) = \ell(v)$ 
for all $h \in H$, and $v \in \pi$.) 

In analogy with works on symplectic periods of automorphic forms on $\GL_{2n}(\A)$, one can consider similar questions by replacing $G= \GL_{2n}$ by $G=\U_{2n}$,
a unitary group defined by a hermitian form on a $2n$-dimensional vector space $V$ over $K$, where $K$ is a quadratic extension of a global field 
$k$.

Observe that 
$$\Sp_{2n}(F) \subset \U_{2n}(F),$$
  when one takes the unitary group in $2n$-variables over $F$ which is quasi-split over $F$, and splits over a quadratic extension $E/F$. 
For example, let 
  $$A = \left ( \begin{array}{ccccccc} 
  && & & & & i \\
& & &  & & i & \\
& & & & i & & \\
& & &  * &  & & \\
& &  -i &  & & & \\
&  -i &  & & & & \\
-i & & & & &  & 
\end{array} \right )
$$
where $i \in E^\times$ with $\bar{i} = -i$. The matrix $A$ is hermitian, but $iA$ is symplectic, and therefore the unitary group
defined by $A$ contains the symplectic group defined by $iA$.  
 
 Since we now have $\Sp_{2n}(F) \subset \U_{2n}(F)$, it is a meaningful question to consider representations on $\U_{2n}(F)$ which are distinguished by $\Sp_{2n}(F)$, or  automorphic representations 
 of $\U_{2n}(\A)$ which have nonzero period integral on  $\Sp_{2n}(k)\backslash \Sp_{2n}(\A)$. In fact this question is already considered by Lei Zhang who proved, cf. Theorem 1.1 in [Zh1] that $(\U_{2n}(F), \Sp_{2n}(F))$ is a Gelfand pair, and hence the space of $\Sp_{2n}(F)$-invariant linear forms on any irreducible admissible representation of $\U_{2n}(F)$ is at most one dimensional, for $F$ any local field. In [Zh2] Zhang further proved that there are no tame supercuspidal representations of $\U_{2n}(F)$ distinguished by 
 $\Sp_{2n}(F)$.

 In this work, 
we prove that 
there is no cuspidal representation  of $\U_{2n}(F)$ distinguished by $\Sp_{2n}(F)$ for $F$ a 
non-archimedean local field --- thus completing the work of Lei Zhang. 
We also prove the corresponding 
global theorem that 
there are no cuspidal representations  of $\U_{2n}(\A_k)$ with nonzero period integral on  $\Sp_{2n}(k) \backslash \Sp_{2n}(\A_k)$
for $k$ any number field or a function field. 

The proof we give for the non-existence of cuspidal representations 
of $\U_{2n}(F)$ distinguished by $\Sp_{2n}(F)$ also proves that there are
no cuspidal representations  of $\GL_{2n}(F)$ distinguished by $\Sp_{2n}(F)$, thus giving another proof of the theorem
of Heumos and Rallis. Our proof  in fact has consequences for representations  of $\SL_{2n}(F)$ distinguished by $\Sp_{2n}(F)$ 
about which we make a general conjecture and prove it in some cases.
We also propose a conjectural answer for the classification of  all representations of 
a quasi-split unitary group with symplectic period.

We completely classify representations of quasi-split unitary group 
in four variables over local and global fields
with nontrivial symplectic periods using
methods of theta correspondence.

 Our analysis with theta correspondence uses relationship of $\U_4(F)$ with  a certain orthogonal group in 6 variables, and symplectic group $\Sp_4(F)$ with a certain orthogonal group in 5 variables; especially 
the first identification seems not so standard, so we have taken some pains to elaborate on these. 
\section{Notation} 
We will use $F$ to denote either a general field, or a local field, and $k$ will be used to denote a global field 
(i.e., a number field or a function field). 
If $F$ is a local field, it will always come equipped
with a fixed non-trivial additive character $\psi: F \rightarrow \C^\times$. For a global field $k$, we will let $\A = \A_k$ denote its adele ring, and we will always fix a non-trivial additive character $\psi_0:\A_k/k \rightarrow \C^\times$.

Given a vector space $V$ over a field $F$, we will let $V^\vee$ denote the dual vector space over $F$. 
If $F$ is a local field with a fixed non-trivial character $\psi: F \rightarrow \C^\times$, 
observe that the dual vector space $V^\vee$ can also be identified to the set of all characters
$\widehat{V}$ of $V$ (the Pontryagin dual):
$$\ell \in V^\vee \longrightarrow \hat{\ell} \in \widehat{V},$$
 defined by $$\hat{\ell}(v) = \psi(\ell(v)).$$

For example, for a symplectic vector space $W=X+ X^\vee$, with $X$ and $X^\vee$ maximal isotropic subspaces of $W$, let $P_X$ be the Siegel parabolic 
in $\Sp(W)$ stabilizing $X$ with unipotent radical $N_X$ which is the vector space of  symmetric 
 elements  $ \{\phi \in {\Hom}(X^\vee,X)| \phi =
\phi^\vee\} \cong {\rm Sym}^2X.$ If we denote the set of symmetric 
 elements of $ {\Hom}(X^\vee,X)$ by  $ {\SHom}(X^\vee,X)$, then 
the natural non-degenerate pairing:

$$\xymatrix{ 
{\Hom(X^\vee, X) \times \Hom(X, X^\vee)} \ar@/^{1.5pc}/[rr] \ar[r] & \Hom(X^\vee, X^\vee) \ar[r]_-{\rm tr} & F,} $$
gives rise to a non-degenerate pairing:
$$\SHom(X^\vee, X) \times \SHom(X, X^\vee) \longrightarrow F,$$
identifying the dual of $\SHom(X^\vee, X)$ to $\SHom(X, X^\vee)$, and therefore, the character group of  
$\SHom(X^\vee, X)$ is identified to $\SHom(X, X^\vee)$ (the identification of course depends on the choice of the non-trivial character 
$\psi: F \rightarrow \C^\times$ which will be fixed throughout the paper).

If $(V, q)$ is a quadratic space over a field $F$, $\OO(V)$ denotes the associated orthogonal group over $F$. 
We will use the notation $\OO(m,n)$, which 
is usually used in the context of real groups, to denote any orthogonal group whose rank over $F$ is $\min\{m,n\}$; the notation $\OO(m,n)$ does not give
full information on the quadratic form, or the isomorphism class of the group, but still carries very useful information specially when dealing with orthogonal groups
which are split or quasi-split, i.e., $\OO(m,n)$ with $|m-n| \leq 2$. If the orthogonal group is $\OO(m,m+2)$, then it is a quasi-split group over $F$, 
split by a unique quadratic field extension  of $F$; for us this quadratic extension will always be $E$, the quadratic extension of $F$ involved in defining
the hermitian form underlying our unitary groups.

We will similarly denote unitary groups by $\U(m,n)$ to be any unitary group whose
$F$-rank is $\min\{m,n\}$. We will use the notation $\OO(m), \U(m)$ to denote any orthogonal or unitary group defined by a quadratic or hermitian space of dimension
$m$, or $\OO_m(F)$, $\U_m(F)$ if we want to be explicit about $F$.

Given a vector space $V$ over $F$ together with a quadratic form $q:V \rightarrow F$, and $a \in F^\times$, we will abbreviate $a\cdot V$ to be the quadratic space
with $V$ as the underlying vector space, and $a\cdot q$ as the quadratic form on $V$. Note that although $\OO(a\cdot V) = \OO(V)$, for  considerations in this
paper  dealing  with the theta correspondence, it will be important to treat $a\cdot V$ as a different quadratic space from $V$ with $a\cdot V$ isomorphic to $V$ if and only
if there is an automorphism $g$ of $V$ such that $q(gv) = a\cdot q(v)$ for all $v \in V$, i.e., $a\cdot V \cong V$ as quadratic spaces if and only if $a$ is a similitude factor for $\GO(V)$. For example, if $E$ is a separable quadratic extension of a field $F$, then
$E$ considered as a two dimensional vector space over $F$  carries the quadratic form $q = \Nm$ where $\Nm(e) = e\bar{e}$. Then for $a \in F^\times$, 
the quadratic space $a \cdot E$ is isomorphic to $E$ if and only if $a \in \Nm(E^\times)$. 

\section{Clifford theory \`a la Bernstein-Zelevinsky}

This section written for the purposes of the next section, develops Clifford theory for smooth 
representations of locally compact totally disconnected
 groups. We recall that Clifford theory in the context of finite groups 
describes irreducible representations of a finite group $G$ in the presence of a normal subgroup 
$N$ of $G$, and takes an especially simple form when $N$ is an abelian normal subgroup, and  $G$ can be written as a semi-direct product $G=N \rtimes H$; see for example, Proposition 25 in  [Se]. 
We have not seen a general form of Clifford theory for smooth representations
of a locally compact totally disconnected group, but Bernstein-Zelevinsky in their analysis of representations of $\GL_n(F)$ restricted to a mirabolic subgroup had to develop such a theory --- at least in this context --- based on rather novel ideas.

Since Bernstein-Zelevinsky's work is written in the specific context of mirabolic subgroups of $\GL_n(F)$, we cannot refer to their theorem, but their method can be adapted to a slightly larger context, which is what we do in this section.

\begin{proposition} Let $G = N \rtimes H$ 
be a locally compact totally disconnected group with $N$ a finite dimensional vector space
over a non-archimedean local field $F$. Let $(\pi,V)$ be a smooth representation of $G$. The group $H$ operates on $N$, and 
hence on $\widehat{N}$, the character group of $N$. For a character $\psi:N\rightarrow \C^\times$, let $$\pi_{N,\psi} = \frac{\pi}{\{n\cdot v - \psi(n)v| n \in N, v \in V\}}, $$
be the twisted Jacquet module of $(\pi,V)$. Observe that $\pi_{N,\psi}$  is a module for $N\rtimes H_\psi$
where $H_{\psi}$ is the stabilizer of $\psi$ in $H$, and that if $\pi_{N,\psi}\not = 0$, 
then so is
$\pi_{N,\psi^h}$, the conjugate of $\psi$ by  any $h \in H$. Assume that for 
$ X= \{\psi \in \widehat{N} | \pi_{N,\psi}\not = 0 \}$, 
there are only finitely many orbits 
of $H$ on $X$. Then there exists a $G$-invariant filtration on $\pi$ whose successive quotients are the representations $\pi_i$ of $G$ 
where the index set $\{i\}$ corresponds to the orbits  of $H$ on $X$,
and the representations $\pi_i$ are 
$$\pi_i \cong {\rm ind}_{N\rtimes H_{\psi_i}}^{N \rtimes H} (\pi_{N,\psi_i}),$$
where $H_{\psi_i}$ is the stabilizer of $\psi_i$ in $H$. (Here, and in the rest of the paper, 
 ${\rm ind}_B^A (V)$ denotes {\it un-normalized}, compactly supported
induction.) 
Further, the open orbits of $H$ on $X$ give rise to submodules of $\pi$, whereas the closed orbits of $H$ on $X$ give rise to quotient representations of $\pi$. 
\end{proposition}

\begin{proof} 
Recall that smooth representations of $N$, a finite dimensional 
vector space over $F$, are described (as for any locally compact totally disconnected group), by nondegenerate representations
of the Hecke algebra ${\mathcal H}(N)$.  Bernstein-Zelevinsky in [BZ] observed that 
because of the isomorphism (of algebras!) afforded by the Fourier transform:
$${\mathcal F}: {\mathcal H}(N) 
\stackrel{\cong}\longrightarrow {\Sc}(\widehat{N}),$$
representations of $N$ 
can  be `geometrized': they are  described as  nondegenerate representations of the 
algebra of Schwartz functions ${\Sc}(\widehat{N})$ on $\widehat{N}$ (an algebra under pointwise 
multiplication).

Thus, a nondegenerate representation $\pi$ of the algebra ${\mathcal H}(N)$ 
gives rise to a sheaf 
${\mathcal E}(\pi)$ 
on $\widehat{N}$ such that ${\mathcal E}_c(\pi)$, the space of compactly supported 
sections of this sheaf on $\widehat{N}$, is equal to $\pi$ as a module for ${\Sc}(\widehat{N})$, and the stalk of the sheaf ${\mathcal E}(\pi)$ 
at a point $x \in \widehat{N}$ is (cf. Proposition 1.14 of [BZ])
$${\mathcal E}_x(\pi) = \pi /\{f\cdot v | f \in {\Sc}(\widehat{N}) {\rm ~with~}  f(x)=0, v \in \pi \}.$$ 

Using the identification of ${\mathcal H}(N)$ with $ {\Sc}(\widehat{N})$, and writing a  point $x \in \widehat{N}$ as $\psi$, it follows that 
$${\mathcal E}_\psi(\pi) = \pi /\{f\cdot v | f \in {\mathcal H}({N}) {\rm ~with~} {\mathcal F}(f)(\psi)= \int_N f(y)\psi(y)=0, v \in \pi \}.$$ 

Therefore 
from an application of what is called the lemma of Jacquet-Langlands about Jacquet modules, cf. lemma 2.33 of [BZ],
the fiber of ${\mathcal E}(\pi)$ at a character $\psi$ of $N$ is nothing but the Jacquet module 
$\pi_{N,\psi}$. 
Thus $ X= \{\psi \in \widehat{N} | \pi_{N,\psi}\not = 0 \}$ is the support of the sheaf ${\mathcal E}(\pi)$.

The sheaf ${\mathcal E}(\pi)$ on $\widehat{N}$ is canonically associated  to $\pi$, hence $\pi$ which is actually
a representation of $G = N \rtimes H$ but is being considered as a representation of $N$ alone for the moment,
becomes a $G$-equivariant sheaf on $\widehat{N}$.

Given any sheaf ${\mathcal E}$ on a locally compact totally disconnected topological space $X$ with a closed
subspace $Z$, we have the well-known Bernstein-Zelevinsky exact sequence:
$$0 \rightarrow \Gamma_c(X-Z,{\mathcal E})
\rightarrow \Gamma_c(X,{\mathcal E})
\rightarrow  \Gamma_c(Z,{\mathcal E}|_Z) \rightarrow 0,$$
(where $\Gamma_c$ refers to compactly supported sections), and this is the exact sequence which is
responsible for the filtration on $\pi$ in the proposition. However, another remark from Bernstein-Zelevinsky 
is needed before completion of the proof of the proposition, which is that if $Z$ is an orbit of characters of 
$N$ under $H$, then 
$\Gamma_c(Z,{\mathcal E}|_Z) $ can be identified to the induced representation which appears in the statement of the proposition. This is nothing but Proposition 2.23 of [BZ].
\end{proof}

The following proposition is the exact analogue of Proposition 25 in  [Se],  a form of Clifford theory, 
except that our 
normal abelian subgroup $N$ is more specific than his.
 (It is actually the previous proposition that we will
use in our work.)
 
\begin{proposition} Let $G = N \rtimes H$ 
be a locally compact totally disconnected
group with $N$ a finite dimensional vector space
over a non-archimedean local field $F$. Let $(\pi,V)$ be an irreducible smooth representation of $G$.
Then the set of characters $\psi: N \rightarrow \C^\times$ of $N$ for which $\pi_{N,\psi} \not = 0$ 
form a single orbit under $H$, 
and $$\pi \cong {\rm ind}_{N\rtimes H_{\psi}}^{N \rtimes H} (\pi_{N,\psi}),$$
where $\psi$ is any character of $N$ for which $\pi_{N,\psi} \not = 0$, 
and $H_{\psi}$ is the stabilizer of $\psi$ 
in $H$. Further, the representation  $\pi_{N,\psi} $ of $H_\psi$ is an irreducible representation, and every
 irreducible smooth representation of $G$ is obtained in this way.
\end{proposition}

\begin{proof} The proof of this proposition follows from the 
observation of Bernstein-Zelevinsky  that smooth representations of $N$,
 a finite dimensional 
vector space over $F$, 
can  be `geometrized' as discussed in the proof of Proposition 1, 
i.e., correspond to compactly 
supported global sections of a $G$-equivariant sheaf of $\Sc(X)$-modules on a locally compact totally disconnected topological space $X = \widehat{N}$.
Clearly (compactly supported) global sections of a $G$-equivariant sheaf $\mathcal E$ on a locally compact totally disconnected
topological space $X$
give rise to an irreducible representation of $G$ if and only if,
\begin{enumerate}
\item the group $G$ operates transitively on $X$,
\item the fiber ${\mathcal E}_x$ of the sheaf ${\mathcal E}$ at any point $x\in X$ is an irreducible representation
of the stabilizer $G_x$ of the point $x \in X$.
\end{enumerate}
The conclusion of the proposition is now clear.
\end{proof}

\section{Non distinction of cuspidal representations}
The aim of this section is to prove that a cuspidal
representation of $\U(n,n)(F)$ is not distinguished by $\Sp(2n,F)$ where $F$ is any non-archimedean local field. 
The proof of this 
result  --- which will assume less than distinction by $\Sp(2n,F)$, and will give more information --- will be 
by an inductive argument on $n$ for which we fix some notation.

Let $W_i$ be the symplectic vector space of dimension $2i$ over $F$ with a fixed 
basis $\langle e_i,\cdots, e_1, f_1,\cdots, f_i\rangle$ 
with symplectic form $\langle -,- \rangle $ with the property that $\langle e_j, f_k\rangle = \delta_{jk} = -\langle f_k, e_j\rangle$, and with all the other products zero. The symplectic spaces $W_i$ form a nested sequence of vector spaces with $W_1 \subset W_2 \subset \cdots \subset W_n$.

Given a symplectic space $W$ over $F$, we have a skew-hermitian space $W_E= W\otimes E$ over $E$ which can be used to 
define a unitary group    $\U(W_E)$ with $\Sp(W) \subset \U(W_E)$. 

For $G =\Sp(W)$ (or $\U(W_E)$), define the Klingen parabolic subgroup $Q$ (resp. $P$)  to be the stabilizer of an isotropic 
 line $\langle w \rangle$ in $W$ (resp. $W_E$). Since any two isotropic vectors in $W$ (or $W_E$) are conjugate under $\Sp(W)$ (or $\U(W_E)$),
the Klingen parabolic subgroups are unique up to conjugacy.

In our analysis below, it will be important  to use the subgroup $Q^1$ 
 of $Q$ (resp. $P^1$ of $P$) stabilizing the  isotropic vector $w$ itself. We call these subgroups 
{\it Klingen mirabolic subgroup} in analogy with the mirabolic subgroup of Bernstein-Zelevinsky for the group 
$\GL_n(F)$. They indeed have  much in common with the mirabolic subgroup of Bernstein-Zelevinsky. If we denote the 
Klingen mirabolic in $\Sp(W_n)$ stabilizing the vector $e_n \in W_n$ by $Q^1_n$, then $Q^1_n = \Sp(W_{n-1}) \cdot H^{2n-2}(F)$, where $H^{2n-2}(F)$ is the Heisenberg 
group on the symplectic vector space $W_{n-1}$ (thus $\dim H^{2n-2}(F)= 2n-1$) with the character group of $H^{2n-2}(F)$ identified
to $W_{n-1}$ such that  the action of $\Sp(W_{n-1})$ on $H^{2n-2}(F)$, and hence on its character group, is the natural action of $\Sp(W_{n-1})$ on $W_{n-1}$.   Similarly, 
if we denote the Klingen mirabolic in $\U(W_n\otimes E)$ stabilizing the vector $e_n \in W_n \otimes E $ by $P^1_n$, then $P^1_n = \U(W_{n-1} \otimes E) \cdot H^{2n-2}(E)$, 
where $H^{2n-2}(E)$ is the Heisenberg 
group on the skew-hermitian vector space $W_{n-1}\otimes E$ (thus $\dim H^{2n-2}(E)= 4n-3$) 
with the character group of $H^{2n-2}(E)$ 
identified
to $W_{n-1}\otimes E$ such that  the action of $\U(W_{n-1}\otimes E)$ on $H^{2n-2}(E)$, and hence on its character group, is the natural action of $\U(W_{n-1}\otimes E)$ on $W_{n-1}\otimes E$.   An essential input for our proof is the fact that 
the Heisenberg group $H^{2n-2}(E)$ contains the Heisenberg group $H^{2n-2}(F)$ as a normal subgroup, and their centers are the same, so $H^{2n-2}(E)/H^{2n-2}(F)$ is a vector space over $F$ which is isomorphic to $W_{n-1}$.

It will be convenient to write out the unipotent radical $N_{n}(G) = H^{2n-2}(E)$ of $P^1_n$, 
as well as the unipotent radical $N_{n}(S) = H^{2n-2}(F)$ of $Q^1_n$ both arising as the stabilizer group of the 
isotropic vector $e_n$ in the matrix form 
with respect to the ordered basis $\langle e_n,\cdots, e_1, f_1,\cdots, f_n\rangle$ as: 
$$ N_n(G)= \left\{ \left(\begin{array}{ccccccc}
 1 & x_{2n-1} & x_{2n-2}    & \cdots & x_{2} & z \\
 0 & 1   & 0 &        & 0 & y_2\\
 0 & 0   & 1      & \cdots & 0& y_{3} \\
 0 &     &        & \ddots &  & \vdots\\
 0 &     &        & \cdots & 1 &  y_{2n-1}\\
 0 &\cdots &      & 0     & 0  &  1 \end{array}\right),  \begin{array}{cccl} x_i & , & y_i & \in E, \,\, z \in F\\
 x_i & = & \bar{y}_i, & 2\leq i \leq n-1, \\
                                                                           x_i & = & -\bar{y}_i, & n \leq i \leq 2n-1. 
\end{array} \right\}$$
and 
$$ N_n(S)= \left\{ \left(\begin{array}{ccccccc}
 1 & x_{2n-1} & x_{2n-2}    & \cdots & x_{2} & z \\
 0 & 1   & 0 &        & 0 & y_{2}\\
 0 & 0   & 1      & \cdots & 0& y_{3} \\
 0 &     &        & \ddots &  & \vdots \\
 0 &     &        & \cdots & 1 &  y_{2n-1}\\
 0 &\cdots &      & 0     & 0  &  1 \end{array}\right),  \begin{array}{ccc}  & x_i, y_i, z \in F  &  \\
 & x_i  =  {y}_i, & 2\leq i \leq n-1,\\
                                                                           & x_i  =  -{y}_i, & n \leq i \leq 2n-1. 
\end{array} \right\}$$

Recall that  $\psi$ is  a fixed non-trivial character of $F$; assuming that $E=F(\sqrt{d}), d \in F^\times$, 
 let $\psi_d$ be the character on trace zero elements of $E$ defined by $\psi_d(e) = \psi(\sqrt{d}e)$, 
and let $\psi_n$ be the character of $N_n(G)$:

$$ \psi_n\left(\begin{array}{ccccccc}
 1 & x_{2n-1} & x_{2n-2}    & \cdots & x_{2} & z \\
 0 & 1   & 0 &        & 0 & y_{2}\\
 0 & 0   & 1      & \cdots & 0& y_{3} \\
 0 &     &        & \ddots &  & \vdots \\
 0 &     &        & \cdots & 1 &  y_{2n-1}\\
 0 &\cdots &      & 0     & 0  &  1 \end{array}\right) = \psi_d(x_{2n-1}+y_{2n-1})= 
\psi(\sqrt{d}[x_{2n-1}+y_{2n-1}]), 
\hspace{3cm}(*)$$
where we note that since $x_{2n-1} = -y_{2n-1}$ for elements in $N_n(S)$, the character $\psi_n$ is trivial on $N_n(S)$,
but since $x_{2n-1} = -\bar{y}_{2n-1}$ for elements in $N_n(G)$, $x_{2n-1}+y_{2n-1} = -\bar{y}_{2n-1} + y_{2n-1}$, therefore 
$\sqrt{d}[x_{2n-1}+y_{2n-1}] \in F$, so the character $\psi_n$ is non-trivial on $N_n(G)$ but trivial on $N_n(S)$.

\begin{proposition} \label{prop3} Let $\pi$ be a smooth representation of the Klingen mirabolic subgroup $P_n^1$ of  $\U(W_n \otimes E)$ which
 is distinguished by the Klingen mirabolic subgroup $Q^1_{n}$ of the symplectic subgroup $\Sp(W_n)$. Then for the
unipotent radical $N_n(G)$ of $P^1_n$, there is a character $\mu:N_n(G)\rightarrow \C^\times$ 
which is either $\psi_{n}$ or trivial such that
$\pi_\mu$, the maximal quotient of $\pi$ on which $N_n(G)$ acts by $\mu$ is a 
smooth representation of the Klingen mirabolic subgroup $P_{n-1}^1$ of  $\U(W_{n-1} \otimes E)$ which
 is distinguished by the Klingen mirabolic subgroup $Q^1_{n-1}$ of the symplectic subgroup $\Sp(W_{n-1})$. 
\end{proposition}

\begin{proof} The proof is a direct consequence of the Clifford theory developed in the last section. Recall that we 
have denoted by 
$N_n(S)$ (resp. $N_n(G)$), the unipotent radical of the Klingen mirabolic in $\Sp(W_n)$ (resp. $\U(W_n \otimes E)$) 
stabilizing the isotropic vector $e_n$. Both  $N_n(S)$ and $N_n(G)$ are normalized by $\Sp(W_{n-1})$, 
and $N_n(S)$ is contained in $N_n(G)$ as a normal subgroup with $N_n(G)/N_n(S) \cong W_{n-1}$  as a module for 
$\Sp(W_{n-1})$.

Let $\pi_{N_n(S)}$ 
be the largest quotient of $\pi$ on which $N_n(S)$ operates trivially. It is a smooth 
module for $\Sp(W_{n-1}) \ltimes N_n(G) /N_n(S) \cong \Sp(W_{n-1}) \ltimes W_{n-1}$. 
Since $\pi$ is distinguished by the mirabolic subgroup $Q^1_n$ of $\Sp(W_n)$, 
$\pi_{N_n(S)}$ is distinguished by $\Sp(W_{n-1})$. 
Since $\pi_{N_n(S)}$ is a module for $\Sp(W_{n-1}) \ltimes W_{n-1}$, we can apply Clifford 
theory to understand this as a module for $\Sp(W_{n-1})$. 

The action of $\Sp(W_{n-1})$ on 
the character group of $W_{n-1}$ has two orbits, one consisting of the trivial character, and the other passing through
the character $\psi_{n}$ whose stabilizer in $\Sp(W_{n-1})$ is the 
Klingen mirabolic subgroup $Q^1_{n-1}$ of the symplectic subgroup $\Sp(W_{n-1})$. Notice that the
character $\psi_{n}$ of $N_n(G)/N_n(S)$ can also be considered as a character of $N_n(G)$, and the stabilizer
of this character of $N_n(G)$ in $\U(W_{n-1} \otimes E)$ is 
the Klingen mirabolic subgroup $P_{n-1}^1$ of  $\U(W_{n-1} \otimes E)$.

By Clifford theory, $\pi_{N_n(S)}$ as a module for $\Sp(W_{n-1})$ has two sub-quotients corresponding to the two orbits 
for the action of $\Sp(W_{n-1})$ on the character group of $N_n(G)/N_n(S) \cong W_{n-1}$. The sub-quotient corresponding to the trivial representation of $W_{n-1}$ being $\pi_{N_n(G)}$, and the other subquotient (in fact a sub-module)
being 
$${\rm ind}_{Q_{n-1}^1}^{\Sp(W_{n-1})}(\pi_{\psi_n})$$
where $\pi_{\psi_n}$ is the maximal quotient of $\pi$ on which $N_n(G)$ acts by $\psi_n$.

Since $\pi_{N_n(S)}$ is distinguished by $\Sp(W_{n-1})$, 
one of these two sub-quotients is distinguished by $\Sp(W_{n-1})$, hence by Frobenius reciprocity either
$\pi_{N_n(G)}$ is distinguished by $\Sp(W_{n-1})$ and therefore also by its Klingen mirabolic subgroup, 
or $\pi_{\psi_n}$ which is a
smooth representation of the Klingen mirabolic subgroup $P_{n-1}^1$ of  $\U(W_{n-1} \otimes E)$ 
 is distinguished by the Klingen mirabolic subgroup $Q^1_{n-1}$ of the symplectic subgroup $\Sp(W_{n-1})$, completing the proof of the proposition. \end{proof}

\begin{corollary}\label{corollary2}
A smooth representation $\pi$ of the Klingen mirabolic subgroup $P_n^1$ of  $\U(W_n \otimes E)$ 
which
 is distinguished by the Klingen mirabolic subgroup $Q^1_{n}$ of the symplectic subgroup $\Sp(W_n)$ carries a
nonzero $\mu_n$-linear form 
for the group of the upper-triangular unipotent matrices in  $\U(W_n \otimes E)$ for $\mu_n$ given by:
$$\mu_n(X)=\psi_d( \epsilon_1[x_{1}+ x_{2n-1}] + \epsilon_2[x_2+ x_{2n-2}] +\cdots +  \epsilon_{n-1}[x_{n-1} + x_{n+1}]), $$
for $$X=\left(\begin{array}{ccccccc}
 1 & x_{1} & *    & * & * & * \\
 0 & 1   & x_2 &   *     & * & *\\
 0 & 0   & 1      & x_3 & * & * \\
 0 &     &        & \ddots & \ddots & \vdots \\
 0 &     &        &  & 1 &  x_{2n-1}\\
 0 &\cdots &      & 0     & 0  &  1 \end{array}\right)$$
where 
the $\epsilon_i$ are either 0 or 1, and we note the most important aspect of the character $\mu_n(X)$ that the term $x_n$ is missing. 
(Recall that $\psi_d(e) = \psi(\sqrt{d}e)$ is a character 
on trace zero elements of $E$.)
\end{corollary}

\begin{proof} Assuming the corollary for $n-1$, it is an immediate consequence of  the proposition that it holds for $n$ 
with $\epsilon_1 =0$ or $1$ depending on the two cases in the proposition; note also that for $n=1$, 
the Klingen mirabolic subgroup of both $\Sp(W_1)$
and $\U(W_1 \otimes E)$ is the group of $2 \times 2$ upper triangular matrices with entries in $F$ for which the corollary is obvious. The form of the character $\mu_n$ follows from the form of the character $\psi_n$ defined before  
Proposition \ref{prop3}.
\end{proof}

\begin{corollary}
 Any representation of $\U(n,n)(F)$ 
distinguished by $\Sp_{2n}(F)$ 
is a sub-quotient of a principal 
series representation of $\U(n,n)(F)$ 
induced from the  Siegel parabolic (with Levi $\GL_n(E)$). In particular, a 
representation of $\U(n,n)(F)$ distinguished by $\Sp_{2n}(F)$ cannot be cuspidal. 
\end{corollary}

\begin{proof}
A representation of $\U(n,n)(F)$ distinguished by $\Sp_{2n}(F)$ is \textsl{a fortiori} distinguished by the Klingen 
mirabolic in $\Sp_{2n}(F)$. It suffices then to observe that the character appearing in Corollary \ref{corollary2} above is trivial on the 
unipotent radical of the Siegel parabolic, hence the Jacquet module corresponding to the Siegel parabolic is nonzero.
\end{proof}

\begin{corollary}
For any representation $\pi$ of  $\U(n,n)(F)$ distinguished by $\Sp_{2n}(F)$, there is a character $\psi: U \rightarrow \mathbb{C}^{\times}$ of the unipotent radical $U$ of a minimal parabolic 
for which $\pi_{U,\psi} \neq 0$. 
\end{corollary}

(By a theorem of Zelevinsky, any representation of $\GL_{2n}(F)$ has this property, but 
this is not the case for other groups, not even for Unitary groups.)

\begin{remark} In this section we have not used any property of a non-archimedean local field, 
and thus the results in this section remain valid for finite fields. In Theorem 2.2.1 of [He], Henderson  has given a complete classification of representations 
of $\U_{2n}(\F_q)$ which are distinguished by $\Sp_{2n}(\F_q)$, 
in particular he proves that there are no cuspidal representations of $\U_{2n}(\F_q)$ which are distinguished 
by $\Sp_{2n}(\F_q)$.
\end{remark}

\begin{remark}The proof given here on distinction of representations of $\U(n,n)(F)$ by $\Sp_{2n}(F)$ remains valid 
almost verbatim for representations of $\GL_{2n}(F)$ distinguished by $\Sp_{2n}(F)$ 
giving another proof of the theorem of
Heumos-Rallis in [HR] on non-existence of cuspidal representations of  $\GL_{2n}(F)$ distinguished by $\Sp_{2n}(F)$. 
In fact the proof given here uses just the Klingen mirabolic subgroup of $\Sp_{2n}(F)$ to draw this conclusion, and 
therefore cannot be expected to give the much finer results which have become available on representations 
of $\GL_{2n}(F)$ distinguished by $\Sp_{2n}(F)$. However, note that our proof uses more of $\Sp_{2n}(F)$, and its Klingen mirabolic subgroup, and almost nothing about the ambient group $\U(n,n)(F)$, or in this case, $\GL_{2n}(F)$, and therefore,
in particular our proof works as well to understand representations of  $\SL_{2n}(F)$ distinguished by $\Sp_{2n}(F)$. We 
only state the following proposition in this regard. \end{remark} 

\begin{proposition}\label{sln}
A smooth representation $\pi$ of $\SL_{2n}(F) = \SL(W_n)$
which
 is distinguished by the symplectic subgroup $\Sp(W_n)$ carries a
nonzero $\mu_n$-linear form for the group of the upper-triangular unipotent matrices in  $\SL(W_n)$ for $\mu_n$ given by:
$$\mu_n(X)=\psi_d( \epsilon_1[x_{1}+ x_{2n-1}] + \epsilon_2[x_2+ x_{2n-2}] +\cdots +  \epsilon_{n-1}[x_{n-1} + x_{n+1}]), $$
for $$X=\left(\begin{array}{ccccccc}
 1 & x_{1} & *    & * & * & * \\
 0 & 1   & x_2 &   *     & * & *\\
 0 & 0   & 1      & x_3 & * & * \\
 0 &     &        & \ddots & \ddots & \vdots \\
 0 &     &        &  & 1 &  x_{2n-1}\\
 0 &\cdots &      & 0     & 0  &  1 \end{array}\right)$$
where the $\epsilon_i$ are either 0 or 1,
and $\psi$ is any (fixed) nontrivial character of $F$.
\end{proposition}

We next recall from Zelevinsky [Ze] the notion of degenerate Whittaker model 
of an arbitrary irreducible smooth
representation $\pi$ of  $\GL_n(F)$. He defines in \S 8 of [Ze] a character $\theta$ on the group $U$ of upper triangular 
unipotent elements of $\GL_n(F)$ by
$$\theta(u_{ij}) = \psi(\sum u_{i,i+1}),$$
where $\sum$ runs over all integers $1,2,\cdots, n-1$ except,
$$n-\lambda_1, n - \lambda_1-\lambda_2,\cdots, n-\lambda_1-\lambda_2-\cdots -\lambda_{k-1},$$
where the integers $\lambda_i$ are inductively defined with $\lambda_1$ being the highest nonzero 
derivative of $\pi$, $\lambda_2$ the highest nonzero derivative of $\pi^{\lambda_1}$, and so on. 
It is a theorem of Zelevinsky (corollary in \S 8.3 of [Ze]) that there is a linear form $\ell:\pi \rightarrow \C$
on which the group $U$ of upper triangular unipotent matrices 
acts by the character $\theta$, 
and the space of such linear forms has dimension 1.

\begin{conj} \label{conj1} Let $\pi$ be an 
irreducible admissible representation of $\GL(W_n)$ which is distinguished by $\Sp(W_n)$. 
Write $\pi$ restricted to $\SL(W_n)$ as a sum of irreducible
representations $\pi = \sum \pi_{\alpha}$ (with multiplicity 1). Then exactly one of the representations
$\pi_\alpha$ is distinguished by $\Sp(W_n)$, and the one which is distinguished by $\Sp(W_n)$ is the one which carries
the invariant linear form $\theta$ of Zelevinsky defined above. (There is a unique representation of $\SL(W_n)$ 
carrying the invariant linear form $\theta$ by the multiplicity one assertion of Zelevinsky for the group $\GL_n(F)$.) 
\end{conj}

\begin{remark} From the classification due to Offen-Sayag 
of irreducible admissible unitary representations of $\GL(W_n)$ which are distinguished by $\Sp(W_n)$, which we will 
recall in section 9,  it follows that
the character $\theta$ of Zelevinsky  is of the form $\mu_n$ introduced in Corollary \ref{corollary2}. Further, observe that
the choice of the character $\psi$ in Conjecture 1 is not relevant since conjugation by the diagonal matrix
$$ a_t= \left(\begin{array}{cccccccclc}
 t & 0 & 0    & 0 & 0 & 0 &  0 & 0 & 0\\
 0 & t^2   & 0 &   0  & 0   & 0 &  0 & 0 & 0 \\
 0 & 0   & t^3 &    0   & 0 & 0 &  0 & 0 & 0\\
 0 &   0  &      0 & \ddots  & 0  & 0 & 0 & 0 & 0 \\
0 &   0  &   0     & 0 &   t^n & 0  &  0 & 0 & 0\\
0 &   0  &    0    & 0 & 0 &  t &  0  & 0 & 0\\
0 &    0 &      0  & 0 & 0 & 0 &  \ddots & 0 & 0\\
0 &   0  &      0  & 0 & 0 & 0  & 0 & t^{n-1} & 0\\
 0 & 0 &    0  & 0   & 0    &  0  &  0& 0 & t^n \end{array}\right) $$
is scaling by $t$ on all simple root spaces except the `middle' one (which is not there in $\mu_n$), so acts transitively
on the set of characters $\mu_n$ arising out of different choices of $\psi$, and $a_t$ being in $\GSp(W_n)$, it preserves
distinction by $\Sp(W_n)$.
\end{remark}
\begin{proposition} Conjecture \ref{conj1} is true for the Speh module $Sp_m(\pi)$ 
(where $\pi$ is a cuspidal 
representation of $\GL_d(F)$ and $m$ is even, so that $Sp_m(\pi)$ has symplectic model) which is the unique
irreducible quotient of the principal series representation 
$\pi\cdot \nu^{(m-1)/2} \times \cdots \times \pi\cdot \nu ^{-(m-1)/2}$
of $\GL_{md}(F)$ (parabolic induction from the representation $\pi\cdot \nu^{(m-1)/2} \boxtimes \cdots \boxtimes \pi\cdot \nu ^{-(m-1)/2}$ of the Levi subgroup 
$\GL_d(F) \times \cdots \times \GL_d(F)$).
\end{proposition}
\begin{proof}It is known that for the Speh module $Sp_m(\pi)$,
 the integers $\lambda_i$ introduced above are  all equal to $d$, and $k=m$.
Thus the character $\theta$ of Zelevinsky is the character 
of the group $U$ of upper triangular unipotent matrices given by $$\theta(u_{ij}) = \psi(\sum u_{i,i+1}),$$
where $\sum$ runs over all integers $1,2,\cdots, n-1$ except,
$$n-d, n - 2d ,\cdots, n-(m-1)d=d.$$

The main point about the Speh module $Sp_m(\pi)$, 
which we will presently prove,  
being that $\theta$ is the only character (up to conjugacy) 
of the unipotent group $U$ for which there is a $\theta$-invariant linear form. Thus the only character which appears in Proposition
\ref{sln} is $\theta$, proving conjecture \ref{conj1} for the Speh modules $Sp_m(\pi)$. 

To prove the assertion regarding characters of $U$ appearing in the Speh
module $Sp_m(\pi)$, 
note that any character of $U$ is of the form
$$\theta_S(u_{ij}) = \psi(\sum a_iu_{i,i+1}),$$
where $a_i \in F$, and   $S$ is defined to be the set of integers $i$ for which $a_i=0$.  Construct the standard parabolic $P=M_SN_S$ of 
$\GL_n(F)$ such that the only simple root spaces in $N$ are $\alpha_i = e_i-e_{i+1}$ for $i \in S$. The character $\theta_S$
is clearly trivial on $N_S$, and therefore the Jacquet module of $\pi$ with respect to $N_S$ is nonzero, and is in fact
generic.  Now we appeal to the `hereditary' property of Jacquet modules for Speh modules: that the  Jacquet
modules of $Sp_m(\pi)$ are themselves product of Speh modules on $\pi$, and therefore the only nonzero generic 
Jacquet module corresponds to the partition $(d,d,\cdots, d)$ of $md$, proving the assertion on the  
characters of $U$ appearing in the Speh
module $Sp_m(\pi)$.
\end{proof}

\begin{remark} In this final remark of the section, we try to delineate the `group theory' which goes into the proof of 
the main result, Proposition \ref{prop3}. 
The paper [AGR] 
calls a pair $(G,H)$ a vanishing pair, if there are no cuspidal representations of $G$ distinguished by $H$. In this paper
we have proved that $(\U_{2n},\Sp_{2n})$ is a vanishing pair. How did we achieve it? To simplify language, let us  be in the context of algebraic groups over finite fields. We need to use the subgroup $H$ to 
construct the unipotent radical $N$ of a parabolic in $G$ such that a cuspidal 
representation $\pi$ of $G$ distinguished by $H$ is also 
distinguished by $N$ leading to a contradiction to cuspidality of $\pi$. Well, begin with the unipotent radical $N(H)$ of a parabolic in $H$. 
Take its normalizer $P_G(N)$ in $G$, and let $N(G)$ be the  unipotent radical of $P_G(N)$, which clearly contains $N(H)$ as a normal subgroup. Since the 
representation $\pi$ we are considering has a $H$-fixed vector, it certainly has $N(H)$-fixed vectors, and $\pi^{N(H)}$ 
is a module for $P_G(N)/N(H)$. 
In our case, $N(G)/N(H)$ is an abelian group, allowing us to understand $\pi^{N(H)}$ 
as a module for $P_G(N)/N(H)$, in particular also for $N(G)$. The group $N(G)$ is nearer to the unipotent 
radical of a parabolic in $G$ (this is a general theorem of Borel-Tits of going from any unipotent group in $G$ 
to the unipotent radical of a parabolic in $G$ by an iterative 
process of the above kind). We do not quite get distinction by $N(G)$, but by a 
character $\chi$ of $N(G)/N(H)$, whose kernel ${\rm ker}(\chi)$ 
is a codimension one subspace of $N(G)$ (containing $N(H)$), so we are making 
progress. The representation   $\pi^{N(H)}$ of $P_G(N)/N(H)$ is distinguished by $P_G(N) \cap H$. This allows one to get some more 
unipotents    from $H$ to be augmented to ${\rm ker}(\chi)$ to reach towards the desired 
unipotent radical $N$ of a parabolic in $G$.
\end{remark}
\section{Non distinction of Cuspidal automorphic representations}

In this section we prove that for cuspidal automorphic functions $f$ on $\U(n,n)(\A_k)$ we must have:
$$\int_{\Sp_{2n}(k)\backslash \Sp_{2n}(\A_k)} 
f(h) dh = 0.$$

Actually we first prove  what appears to be a stronger result, 
that the period integral of cuspidal automorphic functions $f$ on $\U(n,n)(\A_k)$ on Klingen mirabolic $Q^1_n$ 
of $\Sp_{2n}$ is zero:
$$\int_{Q^1_n(k)\backslash Q^1_n(\A_k)} f(h) dh = 0;$$
however,  
vanishing of the symplectic period is not a formal consequence of this. 
For our local theorem, this was no issue: if there are no invariant linear
forms for the Klingen mirabolic, a fortiori, there are none for the larger symplectic group. In the global situation, because
we are dealing with integration on $\Sp_{2n}(k)\backslash \Sp_{2n}(\A_k)$ versus  integration on 
$Q^1_{n}(k)\backslash Q^1_{n}(\A_k)$, we are not quite in a context 
to be able to use Fubini's theorem, and such a conclusion is not obvious, and is effected using an Eisenstein series,
a trick that we learnt from [AGR].

The proof of vanishing of
period integral on Klingen mirabolic, 
will follow closely our local proof. We will also follow exactly the same notation as 
there, thus $W_i$ will be symplectic vector space over $k$ with basis 
$\langle e_i,\cdots, e_1, f_1,\cdots, f_i\rangle$ 
with the symplectic form $\langle -,- \rangle $ with the property that $\langle e_j, f_k\rangle = \delta_{jk} = -\langle f_k, e_j\rangle$, and with all the other products zero. 
The symplectic spaces $W_i$ form a nested sequence of vector spaces with $W_1 \subset W_2 
\subset \cdots \subset W_n$. Given a symplectic space $W$ over $k$, and $K/k$ a quadratic extension, 
we have a skew-hermitian space 
$W_K= W\otimes K$ over $K$ which can be used to 
define a unitary group    $\U(W_K)$ with $\Sp(W) \subset \U(W_K)$.  

 We begin with a {\it global  analogue} of 
the Clifford theory. In fact, in the local theory, one could separate the role of Clifford theory, Mackey theory and the
Frobenius reciprocity, which together allow one to understand when a representation of 
$G = A \rtimes H$ has an $H$-invariant linear form. In the global context, the three steps will merge into one, and we will directly find when an automorphic representation of $G$ has nonzero period integral along $H$.

Let $G = A \rtimes H$ 
be a semi-direct product of algebraic groups over a global field $k$ where $A\cong k^d$ for some integer
$d$. 
Fix $\psi_0: 
\A_k/k\rightarrow \C^\times$ to be a nontrivial character. For any linear map $\ell: A\rightarrow k$, we get an automorphic character
$\psi= \psi_0 \circ \ell: A(\A_k)/A(k) \rightarrow \C^\times$, 
and all automorphic characters on $A(\A_k)/A(k)$ are of this form; i.e.,  characters on $A(\A_k)/A(k)$ are in bijective
correspondence with the dual vector space  $A^\vee(k)$ of the vector space $A$ over $k$.

Let $H_\ell$ be the stabilizer in $H$ of a linear map $\ell: A \rightarrow k$. Then $H_\ell$ is an algebraic subgroup of $H$ defined
over $k$ such that $H_\ell(k) = H_\psi(k)$ is the stabilizer of the 
automorphic character $\psi= \psi_0 \circ \ell: A(\A_k)/A(k) \rightarrow \C^\times$. We will assume in what follows that 
$H(k)\backslash H(\A_k)$, as well as $H_\psi(k)\backslash H_\psi(\A_k)$ have finite measures for all characters 
$\psi= \psi_0 \circ \ell: A(\A_k)/A(k) \rightarrow \C^\times$.

For a function $f$ on $G(k)\backslash G(\A_k)$, 
define its Fourier coefficient $ f_\psi$ to be the function on 
$H_\psi(k)\backslash H_\psi(\A_k)$ 
defined by:
$$f_\psi(h) = \int_{A(\A_k)/A(k)}f(ah) \psi(a)da,$$
where $da$ is a Haar measure on $A(\A_k)/A(k)$. Taking Fourier coefficients gives an $H_\psi(\A_k)$-equivariant map
from smooth functions on $G(k)\backslash G(\A_k)$ to  smooth functions on $H_\psi(k)\backslash H_\psi(\A_k)$:
$${\mathcal F}_\psi: C^\infty(G(k)\backslash G(\A_k)) \rightarrow C^\infty(H_\psi(k)\backslash H_\psi(\A_k)).$$

It will be important to note that ${\mathcal F}_\psi$ 
takes bounded functions in 
$C^\infty(G(k)\backslash G(\A_k))$ 
to bounded functions in $ C^\infty(H_\psi(k)\backslash H_\psi(\A_k))$. Since
${\mathcal F}_\psi$ commutes with $H_{\psi}(\A_k)$, if we have a space $\pi$ of bounded 
functions $C^\infty(G(k)\backslash G(\A_k))$ invariant under differential operators coming from 
$G(k\otimes \R)$, in particular from $H_\psi(k\otimes \R)$, 
the image of ${\mathcal F}_\psi$ under $\pi$ will consist of bounded functions in  
$C^\infty(H_\psi(k)\backslash H_\psi(\A_k))$ invariant under differential operators coming
 from $H_\psi(k\otimes \R)$.
 
\begin{proposition}\label{basic}
With the notation as above (in particular $G=A\rtimes H$, a semi-direct product of algebraic groups over a global field $k$ 
with $A$ a vector space over $k$, and $H_\psi(k)\backslash H_\psi(\A_k)$ have finite measures for all characters 
$\psi= \psi_0 \circ \ell: A(\A_k)/A(k) \rightarrow \C^\times$), 
let $\pi$ be a space of smooth functions on $G(k)\backslash G(\A_k)$ which is $G(\A_k)$-invariant. Suppose $\pi$ 
consists of bounded functions  such that the  period integral on $H(k)\backslash H(\A_k)$ is not identically zero on functions in $\pi$. 
Then there is a function $f \in \pi, $ and a character $\ell: A \rightarrow k$ 
for which the Fourier coefficient $f_\ell = f_\psi$ defined above to be a function on 
$H_\psi(k)\backslash H_\psi(\A_k)$ has nonzero period integral on $H_\psi(k)\backslash H_\psi(\A_k)$.
\end{proposition}
\begin{proof}
Let us  begin with the Fourier expansion:
$$f(ah) = \sum_{\psi:A(\A_k)/A(k) \rightarrow \C^\times} f_\psi(h) \psi(a),$$
where $\psi$ runs over all automorphic characters $\psi: A(\A_k)/A(k) \rightarrow \C^\times$ which as noted earlier 
all arise as $\psi= \psi_0 \circ \ell: A(\A_k)/A(k) \rightarrow \C^\times$ for a linear map $\ell: A \rightarrow k$.

Evaluating the Fourier expansion at $a=1$,
$$f(h) = \sum_{\psi} f_\psi(h),$$
hence,
\begin{eqnarray*}
\int_{H(k)\backslash H(\A_k)} f(h) dh  & = & \int_{H(k)\backslash H(\A_k)}  \sum_{\psi} f_\psi(h) dh \\
& = & \sum_{\psi} \int_{H(k)\backslash H(\A_k)}  
f_\psi(h) dh. 
\end{eqnarray*}

We need to justify interchanging summation and integration above which we shall do  separately in the next two Lemmas so as not to disrupt the flow of argument here.

Combining characters 
$\psi= \psi_0 \circ \ell: A(\A_k)/A(k) \rightarrow \C^\times$ for a linear map $\ell: A \rightarrow k$ 
which are in one orbit for $H(k)$ --- the set of $H(k)$ orbits of such characters being the quotient set 
$A^\vee(k)/H(k)$ --- we find that:
\begin{eqnarray*}
\int_{H(k)\backslash H(\A_k)} f(h) dh  & = & \sum_{\psi} \int_{H(k)\backslash H(\A_k)}  f_\psi(h) dh \\ 
& = & 
\sum_{\psi \in A^\vee(k)/H(k)}\int_{H_{\psi}(k)\backslash H(\A_k)}  f_\psi(h)dh \\
& = & \sum_{\psi \in A^\vee(k)/H(k) }\int_{H_\psi(\A_k)\backslash H(\A_k)} \left [ \int_{H_{\psi}(k)\backslash H_{\psi}(\A_k)}  f_\psi(hh')dh \right ] dh'. \end{eqnarray*}
Therefore if the period integral on $H(k)\backslash H(\A_k)$ is nonzero, so must the inner integral 
$$\int_{H_{\psi}(k)\backslash H_{\psi}(\A_k)}  f_\psi(hh')dh, $$ 
too for some $h' \in H(\A_k)$ and some automorphic character $\psi$ on $A(k)\backslash A(\A_k)$. Since the space of functions in $\pi$ is right invariant under $H(\A_k)$, 
this proves the proposition.
\end{proof}

The following two lemmas justify interchanging summation and integration used above in case $k$ is a 
number field. If $k$ is a function field, cusp forms are known to be locally constant compactly supported functions, thus we will be dealing with 
 finite Fourier expansion in which case interchanging summation and integration is not an issue.

\begin{lemma}Suppose $f(x,t)$ is a function on $X \times T = X \times (\R/\Z)^d$ where $X$ is a measure space. Assume 
that $f(x,t)$ is infinitely differentiable as a function of $t \in T$, for all $x \in X$, $t \in (\R/\Z)^d$. 
Assume that $f$ as well as all its derivatives (with constant coefficients) 
along $(\R/\Z)^d$ are bounded as a function on $X \times T$, and that $X$ has finite measure. Then
$\sum_{\underline{n}} \int_X f_{\underline{n}}(x)dx$ is an absolutely convergent series, and
$$\int_X f(x,0) dx= \sum_{\underline{n}} \int_X f_{\underline{n}}(x)dx,$$
where for $\underline{n}=(n_1,\cdots, n_d) \in \Z^d$, $\underline{t}=(t_1,\cdots, t_d) \in (\R/\Z)^d = T$, 
$f_{\underline{n}}(x)$ 
is the $\underline{n}$-th Fourier coefficient defined by:
$$ f_{\underline{n}}(x) = \int_{(\R/\Z)^d}f(x,\underline{t}) e^{2\pi i (\sum_k n_kt_k)} dt_1\cdots dt_d.$$
\end{lemma} 
\begin{proof} We give a proof only for $d=1$. By the Cauchy-Schwarz inequality, we have:
$$\sum_{n\not = 0}|a_n| \leq \left ( \sum_{n \not = 0}|na_n|^2 \right )^{1/2} \left ( \sum_{n \not = 0} \frac{1}{n^2}\right )^{1/2}.$$
If $a_n$ are the Fourier coefficients of a $C^\infty$-function $f(t)$ on $\R/\Z$, $na_n$ are the Fourier coefficients 
of the function $\frac{df}{dt}$. Therefore using Parseval's (= Plancherel) theorem,

\begin{eqnarray*}
\int_X (\sum_{n\not = 0}|f_n(x)|)dx &\leq&  \sqrt{\frac{\pi^2}{3}} 
\sqrt{ \int_X (\sum_{n \not = 0}|nf_n(x)|^2)dx} \\
& \leq & \sqrt{\frac{\pi^2}{3}} \sqrt{\int_X \left |\frac{df}{dt}(x,t) \right |^2dx dt} \\
& < & \infty,
\end{eqnarray*}
where the last conclusion is arrived at  because $\frac{df}{dt}$ is a bounded function on $X \times T$, and $X$ has finite measure.

Finally, because $X$ is assumed to have finite measure, the Lebesgue dominated convergence theorem allows us to interchange 
summation and integration above.
\end{proof}

Following is the adelic analogue of the previous lemma which can be easily deduced from it, but 
we shall not do so here. In this lemma, we will use the standard notion of a `smooth' function on $\A_k^d$ built out of 
characteristic functions of (translates of) compact open subgroups of finite part of the adele group, and smooth functions at infinity.
 
\begin{lemma}Suppose $f(x,t)$ is a function on $X \times T = X \times (k\backslash \A_k)^d$ where $X$ is a measure space. Assume 
that $f(x,t)$ is `smooth'  as a function of $t \in \A_k^d$, for all $x \in X$, $t \in \A_k^d$. 
Assume that $f$ as well as all its derivatives (with constant coefficients) 
along $(k\backslash \A_k)^d$ are bounded as a function on $X \times (k \backslash \A_k)^d$, and that $X$ has finite measure. Then
$\sum_{y \in k^d} \int_X f_{y}(x)dx$ is an absolutely convergent series, and
$$\int_X f(x) dx= \sum_{y \in k^d} \int_X f_{y}(x)dx,$$
where for $y 
=(y_1,\cdots, y_d) \in k^d $, $\underline{t}=(t_1,\cdots, t_d) \in \A_k^d$, 
$f_{y}(x)$ 
is the Fourier coefficient defined by:
$$ f_{y}(x) = \int_{(k\backslash \A_k)^d}f(x,\underline{t}) \psi(\sum_k y_kt_k) dt_1\cdots dt_d.$$
\end{lemma}

The purpose of the following lemma and its corollary is to have the most obvious relationship between period integrals 
on a group and any of its normal subgroups.

\begin{lemma} Let $G$ be an algebraic  group over a global field $k$, and $N$ a normal algebraic subgroup of $G$ defined
 over $k$. We assume that for all algebras ${\mathfrak a} \supset k$, $(N\backslash G )({\mathfrak a})= N({\mathfrak a})\backslash G({\mathfrak a})$. 
Then for $L^1$-functions $f$ on $G(k)\backslash G(\A_k)$, and for an appropriate choice of right Haar measures, 
we have,
$$ \int_{G(k)\backslash G(\A_k)} f(g) dg = \int_{(N\backslash G)(k)\backslash (N \backslash G)(\A_k)} \left [ \int_{N(k)\backslash N(\A_k)}f(n\bar{g})dn \right ] d\bar{g}.$$
\end{lemma} 
\begin{corollary} \label{cor4}
Let $G$ be an algebraic  group over a global field $k$, and $N$ a normal algebraic subgroup of $G$ defined
 over $k$. Then for a space $V$ of $L^1$-functions on $G(k)\backslash G(\A_k)$ 
which is invariant under right translations by $G(\A_k)$,
if the period integral on $G(k)\backslash G(\A_k)$ is not identically zero on $V$,
then the period integral on $N(k)\backslash N(\A_k)$ is also not identically zero on $V$.
\end{corollary}

The following proposition is the  global analogue of Proposition \ref{prop3} of the last section, with a proof which 
is almost verbatim the proof there.  The notation used in this proposition is accordingly the same as there, in particular
we remind the reader of the character $\psi_n$ introduced before Proposition \ref{prop3}.

\begin{proposition} \label{prop7} Let $\Pi$ be a space of bounded smooth functions on $P_n^1(k)\backslash P_n^1(\A_k)$ 
which is invariant under $P_n^1(\A_k)$ where $P_n^1$
is  the Klingen mirabolic subgroup  of  $\U(W_n \otimes K)$  
consisting of cuspforms (for `standard' parabolics contained in $\U(W_n\otimes K)$: 
notice that even if a standard parabolic 
is not contained in $P_n^1(\A_k)$, its unipotent radical is). Assume that the period integral of $\Pi$ on 
 the Klingen mirabolic subgroup $Q^1_{n}$ of the symplectic subgroup $\Sp(W_n)$ is not identically zero. 
Then for the
unipotent radical $N_n(G)$ of $P^1_n$,  and the automorphic character $\psi_n:N_n(G)(k)\backslash N_n(G)(\A_k) 
\rightarrow \C^\times$, the image of  $\Pi$ under the Fourier coefficient map ${\mathcal F}$ introduced  above,
is a nonzero representation of $P^1_{n-1}(\A_k)$ for $P^1_{n-1}$ the Klingen mirabolic subgroup of  $\U(W_{n-1} \otimes K)$ 
consisting of bounded cuspforms for which the period integral  on 
 the Klingen mirabolic subgroup $Q^1_{n}$ of the symplectic subgroup $\Sp(W_n)$ is not identically zero. 
\end{proposition}
\begin{proof} Let $N_n(S)$ be the unipotent radical of $Q^1_n$, and $N_n(G)$ the unipotent radical of 
$P^1_n$. Since $N_n(S)$ is a normal subgroup of $Q^1_n$, and we are given that $\Pi$ has nonzero period
integral on $Q_n^1(k)\backslash Q_n^1(\A_k)$. Since $N_n(S) \subset Q^1_n$ is a normal subgroup, it follows from Corollary \ref{cor4} that the period integral on $N_n(S)(k)\backslash N_n(S)(\A_k)$ 
is also nonzero on $\Pi$. We consider the  trivial Fourier coefficient of $\Pi$ with respect to $N_n(S)(k)\backslash N_n(S)(\A_k)$,
to construct a space of functions --- call it $\Pi_{N_n(S)}$ --- on $\Sp_{2n-2}(\A_k)\ltimes N_n(G)/N_n(S)(\A_k) = \Sp_{2n-2}(\A_k)\ltimes \A_k^{2n-2}$. 
We now apply Proposition  \ref{basic} with $G = H \ltimes A = \Sp_{2n-2}\ltimes k^{2n-2}$  and $ \pi = \Pi_{N_n(S)}$. Note that there are
two orbits for the action of $\Sp_{2n-2}(k)$ on $k^{2n-2}$, hence also on the character group of $(\A_k/k)^{2n-2}$: 
the zero orbit and the orbit passing through any nontrivial character of $(\A_k/k)^{2n-2}$ such as $\psi_n$.
Observing that 
the character $\psi_n$ is trivial on $N_n(S)(\A_k)$, therefore it defines a character of 
$A(\A_k)= N_n(S)(\A_k) \backslash N_n(G)(\A_k)$, and the corresponding Fourier coefficient on $G$ is the same 
as that on $P^1_n$ because of:
$$\int_{  N_n(G)(k)\backslash N_n(G)(\A_k) }f(n) \psi_n(n) dn = \int_{ A(k)\backslash A(\A_k) }
\left [ \int_{N_n(S)(k)\backslash N_n(S)(\A_k)} f(nn') dn \right ] \psi_n(n')dn'.$$ 

The image of  $\Pi$ under the Fourier coefficient map ${\mathcal F}$ introduced  above consists of cuspforms is an easy result which we leave to the reader; boundedness of functions in ${\mathcal F}(\Pi)$ is clear. 
\end{proof}

\begin{proposition}\label{prop8}
Let $\Pi$ be a cuspidal automorphic representation of $U(W_n \otimes K)$, and let 
$Q^1_n$ be the Klingen mirabolic subgroup in $\Sp(W_n)$. Then if $\int_{Q_n^1(k) \backslash Q_n^1(\A_k)}f(g)dg$ 
vanishes for all $f \in \Pi$, 
$\int_{\Sp_{2n}(k)\backslash \Sp_{2n}(\A_k)}f(g)dg$ vanishes too for all $f \in \Pi$.
\end{proposition}

\begin{proof}
Assuming that $\int_{\Sp_{2n}(k)\backslash \Sp_{2n}(\A_k)}f(g)dg \not = 0$ for some  $f \in \Pi$, we shall prove by contradiction that
$\int_{Q_n^1(k) \backslash Q_n^1(\A_k)}f(g)dg \not = 0$ also for some $f \in \Pi$. Assume if possible that 
$\int_{Q_n^1(k) \backslash Q_n^1(\A_k)}f(g)dg$ 
vanishes for all $f \in \Pi$.

Let $I(s) = {\rm Ind}_{Q_n(\A_k)}^{\Sp_{2n}(\A_k)}(\delta^s)$ be the principal series representation of  $\Sp_{2n}(\A_k)$ 
for  $\delta$ `half the sum of positive roots' for the Klingen parabolic  $Q_n(\A_k)$. If we write the natural decomposition of $Q_n$ as 
$Q_n = {\mathbb G}_m \times Q_n^1$, then for $(t,q) \in Q_n(\A_k) = \A_k^\times  \times Q_n^1(\A_k)$,  
$\delta(t,q) =  |t|^{n}$ where $|t|$ is the usual absolute value on $ \A_k^\times$.

Let $\phi(g,s) \in I(s)$ be  a `standard' section 
of this analytic family of  principal series representations, thus for each $s \in \C$,
$\phi(g,s)$ are functions on $\Sp_{2n}(\A_k)$ such that,
$$\phi(pg,s) = |p|^s\phi(g,s),$$
where $p = (t,q) \in Q_n(\A_k) = \A_k^\times \times Q_n^1(\A_k)$, and $|p|^s = |t|^{ns}$. Let 
$K_\A$ be the maximal compact subgroup of $\Sp_{2n}(\A_k)$ given by 
$K_\A = 
\prod_{v< \infty} \Sp_{2n}({\mathcal O}_v) \times K_\infty$ 
so that 
$ \Sp_{2n}(\A_k) = K_\A \times Q_n(\A)$. 
To say that a family of functions 
$\phi(g,s) \in I(s)$ is  a `standard' section means that 
its restriction to $K_\A$ is a smooth function independent of $s$. By the transformation property,
$\phi(pg,s) = |p|^s\phi(g,s),$ the restriction of $\phi(g,s)$ to $K_\A$ has the property that
$$\phi(pg,s) = \phi(g,s),$$
for all $p \in Q_n(K_\A)= K_\A \cap Q_n(\A)$. 
Conversely, given a smooth function $\phi$ on $K_\A$
with the property $\phi(pg) = \phi(g)$ for all $p \in Q_n(K_\A)= K_\A \cap Q_n(\A)$, there is a unique 
standard section $\phi(g,s) \in I(s)$. 
In particular, there is the unique section $\phi_0(g,s)$ which is identically
1 on $K_\A$, which may be called the standard spherical section of the family $I(s)$.

Now, for a standard section $\phi(g,s) \in I(s)$, consider the Eisenstein series $$E(\phi, g,s)= \sum_{\gamma \in Q_n(k)\backslash \Sp_{2n}(k)} \phi(\gamma g,s),$$ 
a meromorphic family of functions 
on $ \Sp_{2n}(k)\backslash \Sp_{2n}(\A_k)$ which is known to be absolutely convergent for $Re(s)>1$, 
which for the function $\phi(g,s)=\phi_0(g,s)$
has a simple pole at $s=1$,
and $\mbox{Res}_{s= 1}E(\phi_0, g,s) $ is the constant function 1 on 
$\Sp_{2n}(k)\backslash \Sp_{2n}(\A_k)$.  In what follows, we shall denote by $E(g,s)$, the Eisenstein series $E(\phi_0, g,s)$.

As is well known, a cuspform is rapidly decreasing, and an Eisenstein series is slowly increasing. It follows that 
the product of a cusp form (on a group $G$ restricted to a subgroup $H$) with an Eisenstein series on $H$ is 
still rapidly decreasing, and therefore for  $f$ any function in $\Pi$ restricted to  $\Sp_{2n}(k)\backslash \Sp_{2n}(\A_k)$, 
it is meaningful to integrate $f \cdot E(g,s)$ on ${\Sp_{2n}(k)\backslash \Sp_{2n}(\A_k)}$, and unfold the Eisenstein series:

\begin{eqnarray*}
\int_{\Sp_{2n}(k)\backslash \Sp_{2n}(\A_k)} f(g)E(g,s)dg & = & 
 \int_{\Sp_{2n}(k)\backslash \Sp_{2n}(\A_k)}f(g) \sum_{\gamma \in Q_n(k)\backslash \Sp_{2n}(k)} \phi_0(\gamma g,s)dg \\ 
& \stackrel {(*)}= & \int_{Q_n(k)\backslash \Sp_{2n}(\A_k)}f(g)\phi_0(g,s)dg.
\end{eqnarray*}

Using the decomposition, 
$ \Sp_{2n}(\A_k) = K_\A \times Q_n(\A) = K_\A \times \A_k^\times \times Q_n^1(\A)$, we write the Haar measure on $\Sp_{2n}(\A_k)$ as 
$dg = dk dq = dk d^\times \!a dq' $, so that for any $L^1$ 
function $\lambda$ on $\Sp_{2n}(\A_k)$, we have the following form of Fubini's theorem:
\begin{eqnarray*}
\int_{\Sp_{2n}(\A_k)} \lambda(g) dg & = & \int_{Q_n(K_\A) \backslash K_\A} \int_{Q_n(\A_k)} \lambda(k,q)dq dk \\
& = & \int_{Q_n(K_\A) \backslash K_\A} 
\int_{\A_k^\times} \left[ \int_{Q^1_n(\A_k)} 
\lambda(k,a,q')dq' \right ] d^\times\!a dk. 
\end{eqnarray*}
Since $\phi_0(g,s)\equiv 1$ on $K_\A$, it follows from equation $(*)$ that 
\begin{eqnarray*}
\int_{\Sp_{2n}(k)\backslash \Sp_{2n}(\A_k)}f(g)E(g,s)dg  & = & 
\int_{Q_n(K_\A) \backslash K_\A} \left [\int_{\A_k^\times /k^\times} \int_{Q^1_n(k)\backslash Q^1_n(\A_k)}|a|^s f(q' a k) dq' d^\times\!a  \right ]dk \\ 
& = & \int_{Q_n(K_\A) \backslash K_\A} \left [\int_{\A_k^\times /k^\times} |a|^s F( a, k) d^\times\!a  \right ]dk, 
\end{eqnarray*}
where $$F(a,k)=\int_{Q^1_n(k)\backslash Q^1_n(\A_k)}
 f(q' a k) dq', $$
is the integral of a bounded function on a space with finite measure, so the integral is absolutely convergent. Further,
$F(a,k)$ as a function of $a \in {\A_k^\times /k^\times}$ is, by the known property of a cusp form, rapidly decreasing 
at $\infty$ of ${\A_k^\times /k^\times}$, i.e., when $|a|$ tends to infinity.   Therefore, $\int_{\A_k^\times /k^\times} |a|^s F( a, k) d^\times\!a$ is a convergent integral for $Re(s)$ large enough. 

Observe that if the period integral of every function in $\Pi$ on ${Q^1_n(k)\backslash Q^1_n(\A_k)}$ is zero, then the
function $F(a,k)$ will be identically zero, and hence the period integral 
 $\int_{\Sp_{2n}(k)\backslash \Sp_{2n}(\A_k)}f(g)E(g,s)dg$ will be zero at 
least for $Re(s)$ large, and therefore identically 0 as an analytic function.

On the other hand, as mentioned at the end of proof of Proposition 1 in [AGR] as a well-known fact, 
$\mbox{Res}_{s= 1}E(g,s) $ is the constant function 1 on 
$\Sp_{2n}(k)\backslash \Sp_{2n}(\A_k)$, we have 

\begin{eqnarray*} 
\mbox{Res}_{s= 1} \left (\int_{\Sp_{2n}(k)\backslash \Sp_{2n}(\A_k)}f(g)E(g,s)dg \right ) & = & 
\int_{\Sp_{2n}(k)\backslash \Sp_{2n}(\A_k)}f(g)dg,
\end{eqnarray*}
a nonzero number by  our initial assumption that $\int_{\Sp_{2n}(k)\backslash \Sp_{2n}(\A_k)}f(g)dg \not = 0$ for some  $f \in \Pi$, 
 proving that the period integral of some function in $\Pi$ on ${Q^1_n(k)\backslash Q^1_n(\A_k)}$ must be  nonzero.
\end{proof}

\begin{theorem}
Let $\Pi$ be a cuspidal automorphic representation of  $\U(W_n\otimes K)$.
Then the period integral of functions in $\Pi$ on  the Klingen mirabolic subgroup $Q^1_{n}$ 
of the symplectic subgroup $\Sp(W_n)$, as well as on 
the symplectic subgroup $\Sp(W_n)$ is identically zero. 
\end{theorem}
\begin{proof} We first apply Proposition \ref{prop7} to conclude that 
the period integral of functions in $\Pi$ on  the Klingen mirabolic subgroup $Q^1_{n}$ 
of the symplectic subgroup $\Sp(W_n)$ must be identically zero. 

Note that the `boundedness' hypothesis on functions in $\Pi$ in Proposition \ref {prop7} is a well-known consequence of cuspidality. 
The assertion on the period integral of functions in $\Pi$ on  the Klingen mirabolic subgroup $Q^1_{n}$ 
 is a direct consequence of the Proposition \ref{prop7} by an inductive argument 
on noting that for both $\U(1,1)$ and $\Sp(2) = \SL(2)$, the Klingen mirabolic subgroup is the group of upper triangular unipotent matrices, and therefore distinction 
by unipotent group and cuspidality are contradictory to each other. Thus  the period integral of functions in $\Pi$ on 
 the Klingen mirabolic subgroup $Q^1_{n}$ of the symplectic subgroup $\Sp(W_n)$ is identically zero. 

Now, the theorem follows from Proposition \ref{prop8}.
\end{proof}

\begin{remark} The idea of using Eisenstein series in Proposition \ref{prop8} comes from  [AGR], 
specially  their Proposition 2 on page 719. Multiplication  by an Eisenstein series can be interpreted as  the global analogue of the 
identity in the context of finite or locally compact totally disconnected  groups: $\pi \otimes {\rm ind}_H^G 1 
\cong {\rm ind}_H^G(\pi|_H)$ for $\pi$ a 
smooth representation of $G$. Thus $\pi$ carries an $H$-invariant linear form 
if and only if $\pi \otimes {\rm ind}_H^G 1$ 
carries a $G$-invariant linear form; multiplication by Eisenstein series and integration on $G(k)\backslash G(\A_k)$ allows an `unfolding', and 
the proof is achieved along a standard path.
\end{remark}

 \section{Isogenies among classical groups}

The rest of the paper uses theta correspondence to classify irreducible admissible representations of $\U_4(F)$ which are distinguished by $\Sp_4(F)$ both locally and globally. To be able to use methods of theta correspondence, we will find it convenient to turn the pair $(\U_4(F), \Sp_4(F))$ into the closely related 
pair which is $(\SO(4,2), \SO(3,2))$, which we elaborate here for the benefit of some of the readers.
Here $\SO(4,2)$ is a special orthogonal group which is not split, but quasi-split and split over the quadratic extension $E/F$ used to define the 
unitary group $\U(2,2)$, which is also assumed to be quasi-split; the group $\SO(3,2)$ is a split orthogonal group in 5 variables.
\\

\subsection{The isogeny $\Sp(4) \rightarrow \SO(2,3)$}~\\

Let $W$ be a 4 dimensional symplectic space with basis $\{e_1,e_2,e_3,e_4\}$ endowed with the
 symplectic form
$$A = \left ( \begin{array}{ccccc} 
  && & &  1 \\
& & &   1 & \\
&  -1 &   & & \\
-1 & &  &  & 
\end{array} \right ).
$$
The symplectic group $\Sp(W)$ defined using this symplectic form is also the subgroup of $\GL(W)$  fixing the vector 
$w_0 = e_1\wedge e_4 + e_2\wedge e_3$ in $\bigwedge^2W$.

Consider the bilinear form 
$B: \bigwedge^2W \times \bigwedge^2W \longrightarrow \bigwedge^4W \cong F$ given by:
$$(w_1\wedge w_2 , w_3\wedge w_4) \longrightarrow w_1 \wedge w_2 \wedge w_3 \wedge w_4.$$ 
It is easily seen that $B$ is  a non-degenerate symmetric bilinear form on $\bigwedge^2W$ 
 on which $g \in \GL(W)$ operates by scaling by $\det g$, i.e., $g B = (\det g)B$, in particular, $\SL(W)$ preserves the bilinear form, giving rise to a homomorphism from $\SL_4(F)$ to the corresponding orthogonal group in 6 variables which is $\SO(3,3)$.

Further, $$B(w_0,w_0)= 2e_1\wedge e_2 \wedge e_3 \wedge e_4 \not = 0,$$
hence the orthogonal complement $ \langle w_0 \rangle ^{\bot} \subset \bigwedge^2W$ is a non-degenerate quadratic subspace of $\bigwedge^2W$ of dimension 5 preserved by $\Sp(W)$.

This gives rise to an isogeny of algebraic groups $\Sp(4) \rightarrow \SO(2,3)$, making  the following commutative diagram:

$$
\begin{CD}
\Sp(4)
@>  >>  \SL(4)\\
 @ VVV  @VVV  \\
\SO(2,3)
@> >>  \SO(3,3).
\end{CD}
$$ 

\vspace{1cm}

\subsection{The isogeny $\SU(2,2) \rightarrow \SO(4,2)$}~\\

In this section we construct an isogeny from $\SU(2,2)$ to  $\SO(4,2)$, which although is known to exist by generalities (because both groups are quasi-split over $F$ and split by $E$, and the first group is simply connected), we 
have preferred to give an explicit construction in some detail not having found one in the literature (there are some constructions over $\R$). In fact, we were surprised to find that the existence of the isogeny is not there for all hermitian forms (in 4 variables), but only those with discriminant 1, see the remark at the end of the section.

Let $E$ be a quadratic field extension of a field $F$, with $e\rightarrow \bar{e}$ the non-trivial Galois automorphism of $E$ over $ F$. Let $V$ be a vector space over $E$ equipped with a hermitian form $H: V\times V \rightarrow E$
such that:

\begin{enumerate}
\item  $H(v_1d_1, v_2d_2) = \bar{d}_1H(v_1,v_2)d_2$ for all $v_1,v_2 \in V, d_1, d_2 \in E$.

\item $\overline{H(v_1, v_2)} = H(v_2,v_1)$.
\end{enumerate}

\vspace{2mm}
Define $\U(V, H)$ to be the corresponding unitary group which is the isometry group of the pair $(V,H)$, 
and $\SU(V,H)$ to be  the subgroup of determinant one $E$-automorphisms. It will be convenient for us 
to think of $H$ as a $n \times n$ hermitian matrix over $E$ where $n = \dim V$, which we will actually take to be a symmetric matrix over $F$, and define $\U(V,H)$ by:
$$\U(V,H)= \{ g \in \GL(V)| gH\,{}^t\!\bar{g} = H.\}$$

Note that $\GL_4(E)$ operates on the space of $4 \times 4$ skew-symmetric matrices over $E$ by $g\circ X = gX\,{}^t\!g$
which carries a quadratic form, the Pfaffian, given on
$$X = \left ( \begin{array}{cccc} 
 0 & X_{12}& X_{13}&  X_{14} \\
-X_{12}& 0  & X_{23}& X_{24}    \\
-X_{13}&  -X_{23} & 0  & X_{34} \\
-X_{14} & -X_{24}& -X_{34} &  0 
\end{array} \right ),
$$
by (cf. E.Artin's, `Geometric Algebra', page 142)$${\rm Pf}(X) = X_{12}X_{34} + X_{13}X_{42} + X_{14}X_{23} = X_{12}X_{34} - X_{13}X_{24} + X_{14}X_{23}.$$
One knows that ${\rm Pf}(g\circ X) = \det (g) {\rm Pf}(X)$, therefore this gives an explicit homomorphism of $\SL_4(E)$ into
$\SO(3,3)(E)$. In the rest of this section, we will construct a 6 dimensional $F$-subspace of the space of $4 \times 4$ 
skew-symmetric matrices which is left stable by $\SU(V,H)$, and on which Pfaffian takes values in $F$ giving rise to an
isogeny from $\SU(2,2)$ to $\SO(4,2)$.

\begin{lemma} There exists an automorphism $\phi$ of order 2 (well-defined up to $\pm 1$) 
on the space of $4 \times 4$ skew-symmetric matrices over a field $F$
such that,
$$\phi(gX\,{}^t\!g) = \det(g) \,{}^t\!g^{-1}\phi(X)g^{-1}, \hspace{5cm} {(*)}$$
for all $g \in \GL_4(F)$ and $X$ any $4 \times 4$ skew-symmetric matrix over $F$; equivalently, for a 4 dimensional vector space $V$ over $F$, we have a natural isomorphism:
$$ \Lambda^2 V \cong \det(V) \otimes \Lambda^2 V^\vee.$$
Further, the automorphism $\phi$ preserves the Pfaffian: ${\rm Pf}(X) = {\rm Pf}(\phi(X)).$
\end{lemma}

\begin{proof}Identifying the space of $4 \times 4$ skew-symmetric matrices over the field $F$ to $\Lambda^2V$, the mapping
$\phi$ is nothing but what's called the Hodge-$\star$ operator 
(with respect to the quadratic form $X_1^2 + X_2^2+X_3^2+X_4^2$) in general from $\Lambda^kV$ to $\Lambda^{n-k}V$; we omit further details.\end{proof}

\begin{remark} One can write down $\phi$ explicitly as follows:
$$X = \left ( \begin{array}{cccc} 
 0 & X_{12}& X_{13}&  X_{14} \\
-X_{12}& 0  & X_{23}& X_{24}    \\
-X_{13}&  -X_{23} & 0  & X_{34} \\
-X_{14} & -X_{24}& -X_{34} &  0 
\end{array} \right )  \longrightarrow \phi(X)=\left ( \begin{array}{cccc} 
 0 & X_{34}& -X_{24}&  X_{23} \\
-X_{34}& 0  & X_{14}& -X_{13}    \\
X_{24}&  -X_{14} & 0  & X_{12} \\
-X_{23} & X_{13}& -X_{12} &  0 
\end{array} \right ),
$$
and it is thus clear too that ${\rm Pf}(X) = {\rm Pf}(\phi(X)).$
\end{remark}

\begin{lemma} Let $E$ be a quadratic separable extension of a field $F$ with $x\rightarrow \bar{x}$ the Galois involution of $E/F$, and let $H$ be any symmetric non-singular matrix over $F$ with $\det H = 1$. Then the automorphism 
$\phi_H:X\rightarrow \phi(H\bar{X}H)$ of the space of 
$4 \times 4$ skew-symmetric matrices over $E$ is of order 2.
\end{lemma}
\begin{proof}The square of the automorphism $\phi_H: X\rightarrow \phi(H\bar{X}H)$ is 
the automorphism 

$$\xymatrix{ 
X \ar@/^{1.5pc}/[rr] \ar[r]   & \phi(H\bar{X}H) \ar[r] & \phi(H\phi(H{X}H)H) = \det(H)X }.$$
 \end{proof}

\begin{lemma}
The automorphism $X\rightarrow \phi(H\bar{X}H)
$ on the space of 
$4 \times 4$ skew-symmetric matrices over $E$ commutes with the action of the special unitary group $SU(V,H)$ on this space.
\end{lemma}
\begin{proof} We need to prove that:
$$\phi(H\bar{g}\bar{X}\,{}^t\!\bar{g}H) =  g \phi(H\bar{X}H) \,{}^t\!g,$$
but by the defining property $(*)$ of $\phi$ (using that $H$ is a non-singular symmetric matrix over $F$ with $\det H = 1$ 
and $\det g =1$), we have 
\begin{eqnarray*}
\phi(H\bar{g}\bar{X}\,{}^t\!\bar{g}H) & = & H^{-1} \,{}^t\!\bar{g}^{-1}\phi(\bar{X})\bar{g}^{-1}H^{-1}, \\
g \phi(H\bar{X}H) \,{}^t\!g & = & gH^{-1}\phi(\bar{X})H^{-1}\,{}^t\!g.
\end{eqnarray*}
Thus if:
$$H^{-1} \,{}^t\!\bar{g}^{-1}
\phi(\bar{X})\bar{g}^{-1}H^{-1} = gH^{-1}\phi(\bar{X})H^{-1}\,{}^t\!g,$$
we will have proved the lemma. But clearly, this is implied by:
$$H^{-1} \,{}^t\!\bar{g}^{-1} = gH^{-1},$$
which is equivalent to:
$$ \,{}^t\!\bar{g}Hg = H,$$
which is the definition of the unitary group $\U(V,H)$.
 \end{proof}

Note the following general lemma on Galois descent (cf. `The book of involutions' due to Knus et al, Lemma 18.1, 
page 279). 

\begin{lemma}
Let $E$ be a Galois extension of a field $F$, and $W$ a finite dimensional vector space over $E$ equipped with 
a semi-linear action of $G = \Gal(E/F)$ on $W$, 
i.e., there is an $F$-linear action $g \rightarrow \pi(g)$ of $\Gal(E/F)$ on $W$ with $\pi(g)(ew) = g(e) \pi(g)(w)$ for all 
$g \in \Gal(E/F)$, $w \in W$.  
The $F$-subspace $W_0 = W^G$ of $W$ has the property that $W_0\otimes E = W$.
\end{lemma}

It follows from this lemma that the fixed points of the involution $X \rightarrow \phi_H(X) = \phi(H\bar{X}H)$ 
on the vector space $S$ of skew-symmetric matrices over $E$ is a vector space $S_0$ over $F$ 
of dimension 6 with an action of $\SU(V,H)$.

Now $$q(X) = {\rm Pf}(X),$$
the Pfaffian of a  skew-symmetric matrix $X$ over $E$, is an $F$-valued nondegenerate quadratic form on $S_0$ 
which is invariant under $\SU(V,H)$ (since ${\rm Pf}(gX\,{}^t\!g) = \det(g){\rm Pf}(X)$), 
defining the isogeny $\SU(V,H)\rightarrow \SO(S_0)$, 
which for the unitary group defined by the hermitian form:
$$A = \left ( \begin{array}{ccccc} 
  && & &  1 \\
& & &   1 & \\
&  1 &   & & \\
1 & &  &  & 
\end{array} \right ),
$$
lands inside the orthogonal group $\SO(2,4)$ which is  the orthogonal group of the quadratic form of Witt index 2 over $F$ for $ X + E + X^{\vee}$ 
where $X,X^{\vee}$ are maximal isotropic subspaces of $W$ in perfect pairing,
and $E$ is a quadratic separable field extension of $F$ with its $ \mbox{Norm form}$.

The isogeny of algebraic groups $\Sp(4) \rightarrow \SO(2,3)$, together with the inclusion of $\Sp(4) \subset \SU(2,2)$,
gives rise to the following commutative diagram:

$$
\begin{CD}
\Sp(4)
@>  >>  \SU(2,2)\\
 @ VVV  @VVV  \\
\SO(2,3)
@> >>  \SO(2,4).
\end{CD}
$$ 

\begin{remark}The isogeny constructed in this section from $\SU(V,H)$ to an orthogonal group in 6 variables is valid only when one can take $\det H =1$. For instance, over reals, the group $\SU(3,1)$ cannot be isogenous to any one of the groups
$\SO(p,q)$ with $p+q =6$ since an isogeny will also give an isogeny among their maximal compacts, and the maximal compact of $\SU(3,1)$ is $\U(3)$ which is not (isogenous) to the maximal compact subgroup of any one of the $\SO(p,q)$ with
$p+q=6$.
\end{remark} 
\section{Weil representation, and its twisted Jacquet modules}
Let $G$ be a reductive algebraic group over a non-archimedean local field $F$, $P$ a parabolic subgroup of $G$ with Levi decomposition $P=MN$, and
$\psi: N(F) \longrightarrow \C^\times$ a character on $N(F)$. 
In analogy with the Jacquet module, one defines the twisted Jacquet module $\pi_\psi$, for any smooth representation
$\pi$ of $G(F)$ to be the largest quotient of $\pi$ on which $N(F)$ operates by $\psi: N(F) \longrightarrow \C^\times$, i.e.,
$$\pi_\psi = \frac{\pi}{ \{n\cdot v -\psi(n)v | n\in N(F), v \in \pi \} }.$$
These twisted Jacquet modules define an exact functor from smooth representations of $P$ to smooth representations of $M_\psi(F) = \{m\in M(F)| \psi(mnm^{-1})= 
\psi(n), \forall n \in N(F)\},$ i.e., if
$$0\longrightarrow  \pi_1  \longrightarrow  \pi_2 \longrightarrow  \pi_3 \longrightarrow 0,$$
is an exact sequence of smooth $P$-modules, then   
$$0\longrightarrow  \pi_{1,\psi}  \longrightarrow  \pi_{2,\psi} \longrightarrow  \pi_{3,\psi} \longrightarrow 0,$$
is an exact sequence of smooth $M_{\psi}(F)$-modules.

For the dual reductive pair $(\OO(V), \Sp(W))$, we will use  twisted Jacquet modules of the Weil representation of $\Sp(V \otimes W)$ for $P$, a Siegel parabolic
in $\Sp(W)$, and a character $\psi$ on the unipotent radical of such a parabolic subgroup. The twisted Jacquet module is naturally a 
representation of $\OO(V)$, and its structure allows one to relate theta correspondence to distinction of representations.

Before we recall the result on the twisted Jacquet module of the Weil representation, let us   begin by defining the Weil representation itself. 
Let $W = X \oplus X^\vee$ be a symplectic vector space over a
local field $ F$ with $X,X^\vee$ maximal isotropic subspaces in $W$ together with its natural symplectic pairing. Given
a quadratic space $q:V \rightarrow  F$, the Weil representation of $\Sp(V \otimes W)$ gives
rise to a representation of $\OO(V) \times \Sp(W)$ on $\Sc(V
\otimes X^\vee)$, the Schwartz space of locally constant compactly supported functions on $(V \otimes X^\vee)(F)$. 
The Weil representation depends on the choice of a nontrivial additive character $\psi: F \rightarrow \C^\times$ which will be fixed throughout the 
paper. 

Let Us  note that although one talks of Weil representation of $\Sp(V \otimes W)$, it is in fact a representation of a certain two fold (topological) cover 
of $\Sp(V \otimes W)$, called the metaplectic cover of $\Sp(V \otimes W)$, and not of $\Sp(V \otimes W)$ itself. If $\dim V$ is even, then this metaplectic cover
of $\Sp(V \otimes W)$ splits over $\OO(V) \times \Sp(W)$. There is in fact a natural choice of splitting of the metaplectic cover of $\Sp(V \otimes W)$ restricted 
to $\OO(V) \times \Sp(W)$ allowing one to talk of the Weil representation of $\OO(V) \times \Sp(W)$ (for $\dim V$ even).
In this representation, elements of $ \{\phi \in {\Hom}(X^\vee,X)| \phi =
\phi^\vee\} \cong {\rm Sym}^2X,$ which can be identified to the
unipotent radical $N$ of the Siegel parabolic in $\Sp(W)$ stabilizing the isotropic
subspace $X$,
operate on $\Sc(V\otimes X^\vee)$ by
\begin{eqnarray}
(n\cdot f)(x) = \psi((q \otimes q_n)x)f(x), \end{eqnarray}
where $n \in {\Hom}(X^\vee, X)$  gives
rise to a quadratic form $q_n: X^\vee \rightarrow  F,$ which together
with the  quadratic form  $q:V \rightarrow  F$, gives rise to the quadratic
form $q\otimes q_n:V \otimes X^\vee \rightarrow  F$ defined by $(q \otimes q_n)(v \otimes w') = q(v) \cdot q_n(w')$.

The Weil representation realized on $\Sc(V \otimes
X^\vee)$ has the natural action of $\OO(V)$ operating as
\begin{equation*}
L(h)\varphi(x)= \varphi(h^{-1}x).
\end{equation*} The group $\GL(X)$ sits naturally inside
$\Sp(X\oplus X^\vee)$ (preserving $X$ and $X^\vee$),
and its action on $\Sc(V\otimes X^\vee)$ is given by
\begin{equation*}
L(g)\varphi(x)= \chi_V(\det g)|\det g|^{m/2}\varphi(gx),
\end{equation*}
where $m = \dim V$,  
$\chi_V$ is the quadratic character of $ F^\times$ given in terms of
the Hilbert symbol as $\chi_V(a)=(a, {\rm disc} V)$ with ${\rm disc}
V$ the normalized discriminant of $V$. These actions  together with
the action of the Weyl group element (which acts on
$\GL(X)$ sitting inside $\Sp(X \oplus X^\vee)$ through $A\rightarrow
\,{}^t\!A^{-1}$) of $\Sp(W)$ through Fourier
transforms on $\Sc(V\otimes X^\vee)$ --- but which we will not define precisely,
gives the action of $ \OO(V) \times \Sp(W)$ on
$\Sc(V \otimes X^\vee)$.

The Weil representation thus gives rise to a representation of the group 
$\OO(V) \times \Sp(W)$.
Given an irreducible
representation $\pi$ of $\OO(V)$, there exists a representation
$\Theta(\pi)$ of $\Sp(W)$ of finite length, such that $\pi
\otimes \Theta(\pi)$ is the maximal $\pi$-isotypic quotient of $\omega$. It was conjectured by R. Howe that the representation $\Theta(\pi)$ of $\Sp(W)$ has a
unique irreducible quotient $\theta(\pi)$; this conjecture which was proved by Howe in the archimedean case,  by Waldspurger in the non-archimedean case for odd residue characteristic, is now proved in complete generality by W-T. Gan and S. Takeda, cf. [GT]. When one talks about the
theta correspondence, one means the correspondence $\pi \rightarrow
\theta(\pi)$. One can reverse the roles of the groups $\OO(V)$ and $\Sp(W)$
and begin with an irreducible
representation $\pi$ of $\Sp(W)$, and define  a representation
$\Theta(\pi)$ of $\OO(V)$ of finite length, and also the unique irreducible quotient $\theta(\pi)$.

Since $N$, the unipotent radical of the Siegel parabolic of $\Sp(W)$ is a finite dimensional vector space over $F$ 
isomorphic to the space of symmetric elements  in $\Hom[X^\vee, X]$, i.e.,
 $\phi \in \Hom[X^\vee, X]$ such that $\phi^\vee = \phi$, as discussed in the section on Notation, one can
identify the space of  characters $\lambda: N \rightarrow {\mathbb C}^\times$  to symmetric elements
in $\Hom[X,X^\vee]$, i.e., to quadratic forms on $X$,
 through 
the natural non-degenerate pairing:
$$\Hom(X^\vee, X) \times \Hom(X, X^\vee) \longrightarrow \Hom(X^\vee, X^\vee)\stackrel {\rm tr}  \longrightarrow F.$$

Now given a linear map $x:X \rightarrow V$, one can restrict a quadratic form on $V$ to one on $X$; this construction
 plays an important role in the following 
well-known proposition for which we refer to [PR], Corollary 6.2.

\begin{lemma}\label{corollary1} The twisted Jacquet module of the Weil representation corresponding to the
dual reductive pair $(\OO(V), \Sp(W))$ for $N$, the unipotent radical of the Siegel parabolic in $\Sp(W)$ stabilizing $X \subset W$, a maximal isotropic subspace in $W$, is  nonzero exactly for those
characters of $N$ 
which correspond to the `restriction' of quadratic form on
$V$ to $X$ via a linear map $x: X \rightarrow V$.
\end{lemma}

\begin{proposition}\label{prop2}
The twisted Jacquet module of the Weil representation of the dual
reductive pair $(\OO(V), \Sp(W))$  
for $N$, the unipotent radical of 
the Siegel parabolic in $\Sp(W)$ stabilizing $X \subset W$, a maximal isotropic subspace in $W$,  for the
characters of $N$  which corresponds to a
non-degenerate quadratic form on $X$, which we assume is obtained
by restriction of the quadratic form on $V$  via a linear map $x:
X \rightarrow V$ is as a representation of $\OO(V)$ the representation
$${\rm ind}^{\OO(V) } _{ \OO(X^\perp) } {\mathbb C},$$
where $\OO(X^\perp )$ is the orthogonal group of the orthogonal
complement of $X$ inside $V$ sitting inside $\OO(V)$ by acting trivially on $X$.

\end{proposition}

\begin{remark}Assuming that $\Sp(W)= \SL_2(F)$, so that $\dim X = 1$, in which case the previous proposition 
identifies irreducible 
representations $\pi$ of $\OO(V)$ which are distinguished by $\OO(X^\perp)$ 
to theta lifts of (suitable) representations
of $\SL_2(F)$. Observe that if $\pi$ remains irreducible when restricted to $\SO(V)$, therefore
 the representations $\pi$ and $\pi \otimes \det$ of $\OO(V)$ are distinct, $\pi$ restricted to $\SO(V)$ 
is distinguished by $\SO(X^\perp)$ if and only if one of  the representations $\pi$ or
 $\pi \otimes \det$ of $\OO(V)$ is distinguished by $\OO(X^\perp)$  if and only if one of  the representations $\pi$ or
 $\pi \otimes \det$ of $\OO(V)$ arises as a theta lift from (a suitable representation of) $\SL_2(F)$.
\end{remark}

\begin{remark} There are what are called {\it conservation relations}, now proved in all generality in [SZ],
 which for a representation $\pi$ of $\OO(V)$ dictate 
a relationship between first occurrence of $\pi$ in the tower with members $\Sp_{2n}(F)$, with the first occurrence of
$\pi \otimes \det$ in the same tower. If  we are dealing with representations $\pi$ of $\OO(V)$, $\dim V \geq 3$, arising from 
theta correspondence with   $\SL_2(F)$, these conservation relations will force the first occurrence of
$\pi \otimes \det$ to be much later. 
As a result, $\pi$ cannot be isomorphic to $\pi \otimes \det$, equivalently, 
$\pi$ restricted from $\OO(V)$ to $\SO(V)$ must remain irreducible. Thus it is legitimate for us to use theta correspondence between $\SL_2(F)$ and $\SO(V)$ instead of $\SL_2(F)$ and $\OO(V)$.
\end{remark}

\begin{corollary} \label{cor5} Assume that $\Sp(W)= \SL_2(F)$, so that $\dim X = 1$. 
Embed $X$ as a one-dimensional non-degenerate 
subspace $\langle a \rangle \subset V$. Then 
for an irreducible admissible  representation $\mu$ of $\SO(V)$ which is 
distinguished by $\SO(\langle a \rangle ^\perp)$,  the representation $\Theta(\mu)$  of $\SL_2(F)$  has 
a Whittaker model for the character $\psi_a(x) = \psi(ax)$ (in particular, $\theta(\mu) \not = 0$, although because of the difference between $\Theta(\mu)$ and $\theta(\mu)$, $\theta(\mu)$ may not have 
a Whittaker model for the character $\psi_a(x) = \psi(ax)$). Conversely, if an irreducible admissible  representation of $\SO(V)$ is obtained 
as $\theta(\pi)$ for an irreducible admissible 
representation $\pi$ of $\SL_2(F)$ which has a 
Whittaker model for the character $\psi_a(x) = \psi(ax)$, then $\theta(\pi)$ is distinguished by  $\SO(\langle a \rangle ^\perp)$.
\end{corollary}

\begin{remark}
It should be emphasized that in the corollary, we take small theta lift from $\SL_2(F)$ to $\SO(V)$, but big theta 
lift from $\SO(V)$ to $\SL_2(F)$. It is known that the various sub-quotients of the representation $\Theta(\mu)$ 
of $\SL_2(F)$ have the same cuspidal support, and therefore if $\theta(\mu)$ is either cuspidal, or 
is an irreducible principal series, we can replace $\Theta(\mu)$ in the corollary by $\theta(\mu)$. However,
if $\theta(\mu)$ is a component of a reducible principal series, there is a definite possibility of having a difference
between $\Theta(\mu)$ and $\theta(\mu)$ which can affect the conclusion of the corollary 
(if we were to replace $\Theta(\mu)$ by $ \theta(\mu)$). 
\end{remark}

\begin{remark} A consequence of the above corollary is that small theta lift from $\SL_2(F)$ to $\SO(V)$, 
$V$ any quadratic space of dimension $n \geq 4$, of different irreducible (infinite dimensional)
representations of $\SL_2(F)$ which belong to the {\it same} $L$-packet, and therefore have  
Whittaker model for characters $\psi_a(x) = \psi(ax)$, for which $ a \in F^\times/F^{\times 2}$ belong to {\it different} 
cosets, are distinguished by $\SO(\langle a \rangle ^\perp)$;
these subspaces $\langle a \rangle  ^\perp$ have different 
discriminants, and therefore belong to different pure inner-forms of $\SO_{n-1}(F)$. Thus assuming that 
the theta lift of an $L$-packet on  $\OO(V)$ to $\SL_{2}(F)$   makes up a subset of an $L$-packet on 
$\SL_{2}(F)$, we are able to make a contribution to the Gan-Gross-Prasad conjectures for non-tempered representations: that inside an
$L$-packet on $\SO(V)$, there is a unique member which is distinguished by $\SO(W)$ for $W$ a fixed 
codimension one subspace of $V$, i.e.,
multiplicity one holds in such an $L$-packet (and these representations on $\SO(V)$ arise by theta lift from $\SL_2(F)$); further, if instead of $V$ we take the unique other quadratic space $V'$ over $F$ with the same discriminant as $V$, 
then for   $W'=\langle a \rangle ^\perp$, the  orthogonal complement of $\langle a \rangle $ in $V'$, the same analysis proves that a 
theta lift from $\SL_2(F)$ to $\SO(V)$ is distinguished by $\SO(W)$ 
if and only if the theta lift from $\SL_2(F)$ to $\SO(V')$ is distinguished by $\SO(W')$, i.e., 
in the extended Vogan $L$-packet of the pair $(\SO(V),\SO(W))$, 
the multiplicity of distinguished  representations is 2 instead of 1 in the usual Gross-Prasad conjectures (for generic $L$-packets).
\end{remark}

\begin{remark} Corollary \ref{cor5} in various forms has been around in the literature, for example let us  briefly compare it to the work
of Waldspurger [Wa] on toric periods, see e.g., Proposition 14 in [Wa]. In this work of Waldspurger, which is for $V$ a quadratic space of
dimension 3, in which case $\SO(V)$ is either $\PGL_2(F)$ or $\PD^\times$, for the unique quaternion division algebra $D$ over $F$, and $\SO(\langle a \rangle ^\perp)$ is $E^\times /F^\times$ where $E$ is a quadratic algebra over $F$ with the natural embeddings $E^\times/F^\times \hookrightarrow \PGL_2(F)$, and 
$E^\times/F^\times \hookrightarrow \PD^\times$. Waldspurger deals with representations on the metaplectic cover $\overline{\SL}_2(F)$ of $\SL_2(F)$, and concludes as we do, that there is a bijective correspondence between representations of $\PGL_2(F)$ (or of $\PD^\times$) which have a nontrivial toric period for 
$E^\times/F^\times$ with  the corresponding representations of $\overline{\SL}_2(F)$ which have a nontrivial Whittaker functional. For  $\dim(V) =4$, compare Corollary \ref{cor5} with the results of Roberts [Ro], Theorem 7.4 and Corollary 7.5.  
\end{remark}

\section{A lemma on twisted Jacquet modules}
The aim of this section is to fill in a certain detail in Lemma 6.3 of [PT]. For this purpose we first recall that 
lemma (in a suitably modified form).

\begin{lemma}\label{lemma3}
Let $X$ be the $F$-rational points of an algebraic variety defined
over a local field $F$. Let $P$ be a locally compact totally
disconnected group with $P= MN$ for a normal subgroup $N$ of $P$
which we assume is a union of compact subgroups. Assume that $P$
operates smoothly on $\Sc(X)$, and that the action of $P$
restricted to $M$ is given by an action of $M$ on $X$. Suppose that there
is a continuous map from $X$ to characters on $N(F)$, $x \rightarrow \psi_x$,
such that $N$ operates on $\Sc(X)$ by $(n \cdot f)(x) = \psi_x(n)f(x)$.
Fix a character $\psi: N \rightarrow {\mathbb C}^\times$, and let
$M_\psi$ denote the subgroup of $M$ which stabilizes the character
$\psi$ of $N$. The group $M_\psi$ acts on the set of points $x \in
X$ such that $\psi_x = \psi$. Denote this set of points in $X$ by
$X_\psi$ which we assume to be closed in $X$. Then,
$$\Sc(X)_\psi \cong \Sc(X_\psi)$$
as $M_\psi$-modules.
\end{lemma}

The proof of this  lemma in [PT] depends on the
 exact sequence of $M_\psi$-modules,
$$0 \longrightarrow \Sc(X-X_\psi) \longrightarrow \Sc(X) \longrightarrow \Sc(X_\psi) \longrightarrow 0.$$
It is asserted in [PT] that since taking the $\psi$-twisted Jacquet functor is exact, and 
$\Sc(X-X_\psi)_\psi = 0$, 
the lemma follows. However, the fact that $\Sc(X-X_\psi)_\psi = 0$, needs 
an argument which we supply now. 

\begin{lemma}With the same conditions as in Lemma \ref{lemma3}, assume that 
$\psi$ is a character of $N$ which is not 
of the form 
$\psi_x$ for any $x\in X$, then the twisted Jacquet module $\Sc(X)_\psi = 0$.
\end{lemma}
\begin{proof} By twisting the action of $N$ on $\Sc(X)$ by $\psi^{-1}$, it suffices to assume that $\psi = 1$, so that
we are dealing with standard Jacquet modules.

Since $N$ operates on $\Sc(X)$ by $(n \cdot f)(x) = \psi_x(n)f(x)$, it is clear that $N$ leaves $\Sc(X')$ invariant
for any $X'$ which is a compact open subset of $X$. Since $X$ is a union of compact open subsets, $\Sc(X)$ is a 
union (direct limit) of $\Sc(X')$ where $X'$ runs over all compact open subsets of $X$. It is easy to see that to prove
that the Jacquet module $\Sc(X)_N=0$, it suffices to prove that $\Sc(X')_N=0$ for any compact open subset $X'$ of $X$.

To prove that $\Sc(X')_N=0$, we need to prove that 
\begin{eqnarray*}
\Sc(X')[N] &  := & \left \{ f -n\cdot f| n \in N(F), f \in \Sc(X') \right \} \\ 
&=&  \{ (1-\psi_x(n))f(x)| n \in N(F), f \in \Sc(X')\} \\ 
& = & \Sc(X').
\end{eqnarray*}  

It is clear that the subspace of $\Sc(X')$ generated by functions of the form  
$(1-\psi_x(n))f(x)$ where $ n \in N(F)$, and $ f \in \Sc(X')$ 
is an ideal in $\Sc(X')$. 
If this was a proper ideal, it would be contained in 
a maximal ideal, and therefore by the well-known Gelfand-Naimark theorem, all functions in this subspace must vanish 
at some point  $x_0 \in X'$. (We took $X'$ to be compact to be able to apply Gelfand-Naimark theorem; also it 
may be mentioned that although $\Sc(X')$ is not the space of {\it all} continuous functions on $X'$, the conclusion of Gelfand-Naimark theorem --- and its proof --- that the maximal ideals in the space of continuous functions $\C(X')$ 
are in bijective correspondence with points 
of $X'$ is the same for
$\Sc(X')$.) 

For the space of functions generated by 
$(1-\psi_x(n))f(x)$ 
where $ n \in N(F)$, and $ f \in \Sc(X')$,  to vanish at $x_0 \in X'$, we must have 
$(1-\psi_{x_0}(n)) = 0$ for all $n \in N(F)$, which is the same as saying $\psi_{x_0} = 1$, a contradiction 
to our hypothesis that the character $\psi$ (taken to be trivial) is not among the characters $\psi_x, x \in X$, 
 proving that $\Sc(X)_\psi = 0$. 
\end{proof}

\section{Application to distinction of representations}

Let $V = X + E + X^\vee$  
be  a quadratic space of dimension 6 where $X$ and $X^\vee$ are totally isotropic subspaces of $V$ of dimension 2 over $F$ in duality with each other under the
associated bilinear form, and both perpendicular to the space $E$ which is a quadratic field extension of $F$ with its associated norm form ${\mathbb N}m(e) = e\bar{e}$. 
Thus the orthogonal group $\SO(V)$ is a quasi-split orthogonal group which is split by $E$, and may be written as $\SO(4,2)$.

Since $V$ is an isotropic quadratic space, it represents all elements of $F^\times$, i.e., given $a \in F^\times$, there exists $v \in V$ such that $q(v) = a$. On the other hand, 
it is clear that the one dimensional quadratic space $\langle a \rangle $ 
can be embedded inside $(E, \Nm)$ as a quadratic subspace if and only if $a \in F^\times$ is a norm from
$E^\times$. 
It follows that  $\langle a \rangle ^\perp \subset V $ is a split quadratic space if and only if $a\in F^\times$ 
is a norm from $E^\times$, in which case 
$\SO(\langle a \rangle ^\perp)$
could be written as $\SO(3,2)$; if $a \in F^\times$ is not a norm from $E^\times$,  $\SO(\langle a \rangle ^\perp)$
could be written as $\SO(4,1)$ as it is then a quasi-split form of $\SO(5)$ of rank 1 which is split by $E$.

\begin{proposition} Let $\pi$ be an irreducible admissible representation of $\SL_2(F)$ 
which is obtained as a theta lift (for a fixed $\psi:F \rightarrow \C^\times$) from $\OO(2) = \OO(E)$. Then if $\pi$ 
has a Whittaker model for the characters $\psi_a(x) = \psi(ax)$ then $a$ must belong to ${\mathbb N}m(E^\times)$. Conversely, 
an irreducible admissible representation of $\SL_2(F)$ which is dihedral with respect to $E$, i.e., 
is obtained as a  theta lift from
$\OO(b\cdot E)$ for {\it some} $b \in F^\times$,
and  has a Whittaker model 
for a character $\psi_a(x) = \psi(ax)$ for $a \in {\mathbb N}m(E^\times)$, then it is obtained as a theta lift from $\OO(2) = \OO(E)$.
\end{proposition}

\begin{proof} From equation $(1)$,
$$(n\cdot f)(x) = \psi((q \otimes q_n)x)f(x), $$
with $x \in E$, $(q \otimes q_n)(x)= 
n^2 {\N}m(x)$, so the first part of the proposition follows. For the second part of the proposition, observe that by the first part of the proposition, 
if a representation of $\SL_2(F)$ is obtained as a theta lift of a representation of $\OO(b\cdot E)$, then it has Whittaker model only 
for characters of the form $\psi_{bc}(x) = \psi(bcx)$ for some $c \in \Nm E^\times$. Since it is given that $\pi$ has a Whittaker model for 
$a \in \Nm E^\times$, it follows that $b \in \Nm E^\times$. Since $b \in \Nm E^\times$, it follows that $b\cdot E \cong E$ as quadratic spaces, and
hence $\pi$ is indeed obtained as theta lift from $\OO(E)$ as desired. 
\end{proof}

\begin{proposition}\label{prop11}
For an irreducible admissible representation $\pi$ of $\SL_2(F)$, the following are equivalent:

\begin{enumerate}
\item $\pi$ has a Whittaker model for a character $\psi_a(x) = \psi(ax)$ for some $a \in {\mathbb N}m(E^\times)$.  

\item If $\pi$  is obtained as a theta lift of a representation of $\OO(2) = \OO(b \cdot E)$ for some $b\in F^\times$, 
it is obtained as a theta lift from $\OO(E)$; equivalently, $b \in \N E^\times$ so that $b \cdot E \cong E$ as quadratic spaces.
\end{enumerate}
\end{proposition}
\begin{proof}We give a proof by a  case-by-case analysis.

\begin{enumerate}
\item The representation $\pi$ is contained in an irreducible representation $\tilde{\pi}$ of $\GL_2(F)$ which remains irreducible when restricted to $\SL_2(F)$, 
i.e.,  $\tilde{\pi}|_{ \SL_2(F)} = \pi$.  
In this case, $\pi$ and $\tilde{\pi}$ have Whittaker model for {\it all} (non-trivial) characters of $F$, so nothing to be done in this case, i.e., 
$(1)$ is true, and $(2)$ is vacuously true.

\item The representation $\pi$ is contained in an irreducible representation $\tilde{\pi}$ of $\GL_2(F)$ which decomposes into 2 or 4 components when restricted
to $\SL_2(F)$, but $\tilde{\pi}$ does not arise from a character of $E^\times$. Let $L$ be the compositum of all quadratic extensions $M$ of $F$ such that
$\tilde{\pi}$ is a dihedral representation corresponding to a character of $M^\times$.
Then $L$ is either a  quadratic or bi-quadratic 
extension of $F$ such that $\pi$ has Whittaker model exactly for those characters of the form $\psi_a(x)=\psi(ax)$ for $a$ belonging to a fixed 
coset of $F^\times/{\N}m(L^\times)$.
It is easy to see that since $E$ is not contained in $L$, such a coset must intersect ${\N}m(E^\times)$, i.e., the map:
$$\Nm(E^\times) \longrightarrow  F^\times/ {\N}m L^\times,$$
must be surjective, i.e., 
$F^\times = {\N}m E^\times \cdot {\N}m L^\times$. 
But ${\N}m E^\times,$ is a subgroup of $F^\times$ of index 2, therefore $F^\times = {\N}m E^\times \cdot {\N}m L^\times$ if and only if,
$$\Nm L^\times \not \subset  {\N}m E^\times.$$
But by classfield theory, 
  $$\Nm L^\times  \subset  {\N}m E^\times \Longleftrightarrow E^\times \subset L^\times.$$
In this case by hypothesis, 
$E^\times \not \subset L^\times$, so 
the map:
$\Nm(E^\times) \longrightarrow  F^\times/ {\N}m L^\times$ is surjective.

It follows that in this case $\pi$ always has a Whittaker model for a character $\psi_a(x) = \psi(ax)$ for some $a \in {\mathbb N}m(E^\times)$, and $(2)$ is vacuously satisfied (since in this case $\pi$ is not obtained as a theta lift from $\OO(2)$).

\item  The representation 
$\pi$ 
is obtained as a  theta lift from
$\OO(b\cdot E)$ for {\it some} $b \in F^\times$. In this case, the conclusion is part of the previous proposition.
\end{enumerate}\end{proof}

\begin{theorem}
 An irreducible admissible representation of $\SO(X + E + X^\vee) = \SO(4,2)$ is distinguished by $\SO(3,2)$ if and only if it is
obtained as a theta lift of a representation $\pi$ of  $\SL_2(F)$ which has either of 
the following equivalent properties:

\begin{enumerate}
\item $\pi$ has a Whittaker model for a character $\psi_a(x) = \psi(ax)$ for $a \in {\mathbb N}m(E^\times)$.  

\item If $\pi$  is obtained as a theta lift of a representation of $\OO(2) = \OO(b \cdot E)$ for some $b\in F^\times$, 
it is obtained as a theta lift from $\OO(E)$.
\end{enumerate}

\end{theorem}
\begin{proof} By Corollary \ref{cor5}, we already know that an irreducible admissible representation of $\SO(X + E + X^\vee) = \SO(4,2)$ is 
distinguished by $\SO(3,2)$ if and only if it is
obtained as a theta lift of a representation $\pi$ of  $\SL_2(F)$ which 
has a Whittaker model for a character $\psi_a(x) = \psi(ax)$ for some $a \in {\mathbb N}m(E^\times)$.  
(Observe that by the theorem on `stable range', since the split rank of $\SO(4,2)$ is 2, every irreducible admissible representation
of $\SL_2(F)$ has a nonzero theta lift to $\SO(4,2)$.)

Equivalence of $(1)$ and $(2)$ is the content of the previous proposition. \end{proof} 

\begin{corollary}
 An irreducible admissible supercuspidal representation of $\SO(X + E + X^\vee) = \SO(4,2)$ cannot be  distinguished by $\SO(3,2)$. A supercuspidal representation of $\SO(X + E + X^\vee) = \SO(4,2)$  which is obtained 
as a theta lift from $\SL_2(F)$ is distinguished  by $\SO(4,1)$.
\end{corollary}

\begin{proof}
To prove the corollary it suffices to note that by Theorem 9.1, a  supercuspidal representation of $\SO(4,2)$ distinguished by $\SO(3,2)$ must  
be obtained as a theta lift of a 
representation of $\SL_2(F)$ 
which has a Whittaker model for the character $\psi_a(x)=\psi(ax)$ for some $a \in {\mathbb N}m(E^\times)$. 

By Proposition 3,  
a representation $\pi$ of $\SL_2(F)$ 
which has a Whittaker model for the character $\psi_a(x)=\psi(ax)$ for some $a \in {\mathbb N}m(E^\times)$ is either

\begin{enumerate}
\item obtained as a theta lift from
$\OO(2) = \OO(E)$, and therefore by the Kudla's theory of towers of theta lifts, the theta lift of such a representation of $\SL_2(F)$ to $\OO(X + E + X^\vee)= \OO(4,2)$ cannot be supercuspidal, or

\item the representation $\pi$ is not obtained as a theta lift from $\OO(bE)$ for any $b \in F^\times$. 
In this case, the first
occurrence of $\pi$ in the tower $\OO(V_{b,r}) = \OO(Y_r+ bE+ Y_r^\vee)$, where $Y_r$ has dimension $r$ and $b \in F^\times$, and hence $V_{b,r}$ 
has dimension $2 + 2r$, has $\dim(V_{b,r}) \geq 4$ for any $b \in F^\times$. Since the sum of the first occurrences in the two towers is 8 by the `conservation relations', $\pi$ lifts
to both the towers for $\dim (V_{b,r}) = 4$, in particular $\pi$ lifts to $\OO( Y_r + E +  Y_r^\vee)$ for $\dim Y_r =1$, i.e., to $\OO(3,1)$. Again, the lift of $\pi$ to $\OO(4,2)$ cannot be supercuspidal.
\end{enumerate}
    
For the second assertion contained in the corollary regarding distinction  by $\SO(4,1)$, note that by the 
previous analysis, the only supercuspidal representation of $\SO(4,2) $  which is obtained as a 
theta lift from a representation $\pi$ of $\SL_2(F)$ has the property that $\pi$ is obtained as a theta lift 
from $\OO(b\cdot E)$ for $b \in F^\times - \Nm E^\times$. By the conservation relations, theta lift of such representations
to $\SO(4,2)$ are indeed supercuspidal (being the first occurrence), and by Corollary \ref{cor5}, these representations of 
$\SO(4,2)$ are distinguished by $\SO(W)$, where $W$ is the orthogonal complement of 
$b\cdot E$ inside the quadratic space $X+ E  + X^\vee$ with $\dim X =2$ (it is easily seen that 
$b\cdot E$ is contained in the quadratic space $X+ E  + X^\vee$).
Such a $W$ can be seen to be the unique non-split quadratic space of dimension 5 with trivial discriminant, thus $\SO(W)=\SO(4,1)$.
\end{proof}

We will not go into any details of the corresponding global theorem except to state the following theorem which is a simple consequence of  Theorem 11 of [PT].

\begin{theorem}
For a  cuspidal automorphic representation $\pi$ of  $\SL_2(\A_k)$ which 
has a Whittaker model for a character $\psi_{0,a}(x) = \psi_0(ax)$ for $a \in {\mathbb N}m(K^\times)$, its theta lift $\Theta(\pi)$
to $\SO(X + E + X^\vee) = \SO(4,2)(\A_k)$ 
 has convergent, and nonzero period integral on  $\SO(3,2)(k)\backslash \SO(3,2)(\A_k)$.  Conversely, if a 
cuspidal 
automorphic representation of $\SO(X + E + X^\vee) = \SO(4,2)$ has nonzero period integral on  $\SO(3,2)(k)\backslash \SO(3,2)(\A_k)$, 
it is
obtained as a theta lift of a cuspidal automorphic representation $\pi$ of  $\SL_2(\A_k)$ which
 has a Whittaker model for a character $\psi_{0,a}(x) = \psi_0(ax)$ for $a \in {\mathbb N}m(K^\times)$. 
\end{theorem}

Having done many explicit examples as well as some general theorems in this section, we end with the following conjecture.

\begin{conj} Let $W\subset V$ be non-degenerate quadratic spaces over any local field $F$ with $\dim(V/W)=1$. Then there are
 tempered representations of $\SO(V)$ distinguished by $\SO(W)$ if and only if rank$ (\SO(W))$, i.e. the dimension of a maximal isotropic subspace
of $W$,  is $\leq 1$.
\end{conj}

\begin{remark} Over a non-archimedean field $F$, there are not many quadratic pairs 
$W\subset V$, $\dim(V/W)=1$, with rank$ (\SO(W)) \leq 1$. The largest possible 
$W$ with these properties has $\dim (W) = 6$. In this paper in Corollary 6 we have constructed
supercuspidal representations of $\SO(4,2)$ distinguished by $\SO(4,1)$. In the paper [Va], Mahendra Varma 
has constructed supercuspidal representations of $\GL_2(D)$ distinguished by a rank 1 form of $\Sp_4(F)$ (denoted there as $\Sp_2(D)$), 
which can be interpreted as a representation of $\SO(5,1)$ distinguished by $\SO(4,1)$. 
For $F$ any local field, the rank 1 group $G=\SO(n+1,1)$ 
which has a minimal parabolic $P$ with Levi subgroup 
$\SO(n) \times \SO(1,1)$, containing the rank 1 subgroup $H=\SO(n,1)$ has an $H$-open orbit on $P\backslash G$
of the form $\SO(n)\backslash \SO(n,1)$. Thus there are many tempered principal series representations of $G=\SO(n+1,1)$ distinguished by $\SO(n,1)$. Similarly, for $F$ any local field, the rank 2 group $G=\SO(n,2)$ 
which has a maximal parabolic $P$ with Levi subgroup 
$\SO(n-2) \times \SO(2,2)$, containing the rank 1 subgroup $H=\SO(n,1)$ has an $H$-open orbit on $P\backslash G$
of the form $(H \cap P)\backslash H$ with the projection of $H\cap P$ onto the Levi of $P$ to be 
$\SO(n-2) \times \SO(1,1)$. Since there are many tempered representations of $\SO(2,2)$ distinguished by $\SO(1,1)$,there are many tempered principal series representations of $G=\SO(n,2)$ distinguished by $\SO(n,1)$. 

Thus the non-obvious part of the conjecture above is to say that there are no distinguished tempered representations
in the cases not allowed by the conjecture. In the non-archimedean case, there is also the question if there are distinguished supercuspidal representations for the pair:  $(\SO( 5,2),\SO(5,1))$. We are also not sure about detailed analysis for the pairs $(\SO(3,2),\SO(2,2))$ and 
$(\SO(3,2), \SO(3,1))$.
\end{remark}

\section{Interpretation via Langlands parameters}

We begin with the following most natural conjecture regarding distinction of representations of unitary groups by the symplectic group, for which 
we indicate a proof for the case of $\U(2,2)$ dealt with in this paper.

\begin{conj}For $F$ a local field,  let $\{\pi\}$ be an $L$-packet of  irreducible admissible representations of $\U(n,n)(F)$ 
which we assume to be the $L$-packet associated to an Arthur packet on $\U(n,n)(F)$. Then some member of the set $\{\pi\}$ 
 is distinguished by $\Sp_{2n}(F)$ if and only if
under basechange, the representation $BC(\pi)$ of $\GL_{2n}(E)$ is distinguished by $\Sp_{2n}(E)$.
\end{conj}

\begin{remark} Given the classification of representations of $\GL_{2n}(E)$ which are distinguished by $\Sp_{2n}(E)$ --- which we will recall below --- a consequence of the above conjecture is that there should be no tempered representations of $\U(n,n)(F)$ which are distinguished by $\Sp_{2n}(F)$. Recall that in an earlier section, 
we have proved that there are no cuspidal representations of $\U(n,n)(F)$ which are distinguished by $\Sp_{2n}(F)$. 
\end{remark}

We next recall the theorem of Offen-Sayag about symplectic periods of representations on $\GL_{2n}(F)$ in terms of Langlands parameters.

Let $W'_F = W_F \times \SL_2(\C)$ be the Weil-Deligne group of $F$. 
Let $W''_F= W'_F \times \SL_2(\C) = W_F\times \SL_2(\C) \times \SL_2(\C)$. There is a natural homomorphism
$\iota: W'_F \rightarrow W''_F = W'_F \times \SL_2(\C)$ 
in which the mapping from $W'_F$ to itself is the identity map, and the mapping from
$W'_F= W_F \times \SL_2(\C)$ to $\SL_2(\C)$ is trivial on $\SL_2(\C)$, and on $W_F$ is given by
$$w\longrightarrow \left ( \begin{array}{cc} \nu^{1/2} & 0 \\ 0 & \nu^{-1/2} \end{array} \right ),$$
where $\nu$ is the character of $W_F$ (thus factoring through $F^\times$) which is unramified, and takes
 a uniformizer in $F^\times$ to $q^{-1}$ where $q$ is the cardinality of the residue field.

The mapping $\iota: W'_F \rightarrow W''_F = W'_F \times \SL_2(\C)$ 
allows one to restrict admissible homomorphisms in $\Hom[W''_F, \GL_m(\C)]$ 
(whose restriction to $W_F$ have bounded image)
 to admissible homomorphisms in  $\Hom[W'_F, \GL_m(\C)]$ which are certain Langlands parameters
 of irreducible admissible unitary representations of $\GL_m(F)$. 
Admissible representations of $W''_F$ (whose restriction to $W_F$ have bounded image) are called Arthur
parameters, and their restriction to $W'_F$ via the mapping $\iota: W'_F \rightarrow W''_F $ is called the 
Langlands parameter associated to an Arthur parameter. (By the work of Moeglin-Waldspurger, such Langlands parameters account for all representations of $\GL_m(F)$ 
which arise in the theory of automorphic forms.)

Let $\St_n$ denote the unique irreducible $\C$-representation of $\SL_2(\C)$  of dimension $n$. 

\begin{theorem}(Offen-Sayag) \label{OS} Let $\pi$ be the irreducible admissible unitary representation of $\GL_{2n}(F)$ with Langlands
 parameter $\sigma_\pi\circ \iota: W'_F \rightarrow \GL_{2n}(\C)$ for  
an admissible representations $\sigma_\pi$ 
of $W''_F = W'_F \times \SL_2(\C)$ of dimension $2n$ written in the form:
$$\sigma_\pi = \sum_i \sigma_i \otimes \St_i,$$
where $\sigma_i$ are admissible (bounded) representations of $W'_F$. Then the representation $\pi$ has a symplectic model if and only if $\sigma_i=0$ for $i$ an odd integer.
\end{theorem}

It follows that the Langlands parameters of representations of $\GL_{2n}(F)$ with symplectic period have the shape:
$$\sigma_\pi = \sum_i \sigma_i \otimes [\nu^{(2i-1)/2} + \nu^{(2i-3)/2} + \cdots + \nu^{-(2i-3)/2} + \nu^{-(2i-1)/2}],$$
where $\sigma_i$ are `tempered' parameters of $W'_F$.

Suppose now that we are considering representations of $\GL_{2n}(E)$ with symplectic period which arise by basechange from 
representations of $\U(n,n)(F)$. The Langlands parameter of such representations are conjugate-selfdual, and therefore in the
decomposition:
$$\sigma_\pi = \sum_i \sigma_i \otimes [\nu^{(2i-1)/2} + \nu^{(2i-3)/2} + \cdots + \nu^{-(2i-3)/2} + \nu^{-(2i-1)/2}],$$
the representations $\sigma_i$ of $W'_E$ are also conjugate-selfdual.

By the calculation done in [GGP], the component group of such parameters of $\U(n,n)(F)$ are trivial, i.e., the $L$-packet of such representations 
of $\U(n,n)(F)$ consists of single elements (because of the presence of non-trivial powers of $\nu$ in $\sigma_i \otimes \nu^{j/2}$ which appear in $\sigma_{\pi}$, 
none of these can be conjugate-selfdual). We note this as a proposition.

\begin{proposition}
For $F$ a local field,  let $\{\pi\}$ be an $L$-packet of  irreducible admissible representations of $\U(n,n)(F)$ 
which we assume to be the $L$-packet associated to an Arthur packet on $\U(n,n)(F)$. Then 
if under basechange, the representation $BC(\pi)$ of $\GL_{2n}(E)$ is distinguished by $\Sp_{2n}(E)$, the $L$-packet $\{\pi\}$ 
must consist of a single member.
\end{proposition}

In the rest of this section, we indicate how our work in this paper is in conformity with Conjecture 3 
in the case of $\U(2,2)$.

Recall that the $L$-group of the quasi-split group $\SO(4,2)$ over $F$ which is split by the quadratic extension $E$ of $F$ can be taken to be $\OO(6,\C)$,
such that a Langlands parameter for $\SO(4,2)$ consists of an admissible homomorphism $\sigma: W'_F \longrightarrow \OO(6,\C)$ with $\det \sigma = \omega_{E/F}$,
the quadratic character of $F^\times$ associated by classfield theory to the extension $E/F$.

It follows from the formalism of theta lifts that if the Langlands parameter of the representation $\pi$ of $\SL_2(F)$ 
is $\sigma_\pi: W'_F\rightarrow \PGL_2(\C) = \SO(3,\C)$, 
the
Langlands parameter of the representation $\theta(\pi)$ of $\SO(4,2)$ is the following representation of $W'_F$:
\begin{eqnarray}  \omega_{E/F}\sigma_\pi 
+\St_3,  \end{eqnarray}
where we have denoted by $\St_3$ the 3-dimensional representation of $W_F$ which is $[\nu^{-1} + 1 + \nu]$ (thus 
the present $\St_3$ is what would be denoted earlier by $\St_3 \circ \iota$).


On the other hand, for a conjugate-symplectic parameter $\lambda: W'_E \rightarrow \GL_4(\C)$, arising from a representation of $\U_4(F)$,
$\det(\lambda)^{-1/2} \Lambda^2(\lambda)$ is a 6-dimensional
representation with values in $\OO_6(\C)$, where $\det(\lambda)^{1/2}$ 
is a character of $W_E$ whose square is $\det(\lambda)$, and the square root  must exist if the representation of $\U_4(F)$ can be related to one of $\SO_6(F)$ (since there is a homomorphism from $\SU_4(F)$ to $\SO_6(F)$ with kernel $\pm 1 \subset \SU_4(F)$, 
only those representations of $\SU_4(F)$ descend to representations of $\SO_6(F)$ which are trivial on $\pm 1 \subset \SU_4(F)$). 

Via the isogeny from $\SU(2,2)$ to $\SO(4,2)$ of section 6.2, the correspondence of 
representations of $\U(2,2)$ and $\SO(4,2)$ associates to a  conjugate-symplectic parameter $\lambda: W'_E \rightarrow \GL_4(\C)$, 
arising from a representation of $\U_4(F)$, a parameter $W'_F \rightarrow \OO_6(\C)$ associated to a representation of $\SO(4,2)(F)$, 
whose basechange to $E$ is 
$\det(\lambda)^{-1/2} \Lambda^2(\lambda) : W'_E \rightarrow \OO_6(\C)$. 

Note that if $\lambda = \sigma \otimes \St_2$ is a conjugate-symplectic representation of $W'_E$ 
(the only non-trivial option allowed by the theorem of Offen-Sayag which we are applying after basechanging 
the representation of $\U(2,2)(F)$ to $\GL_4(E)$), 
then,
$$\Lambda^2(\sigma \otimes \St_2) = \Lambda^2(\sigma) \otimes {\rm Sym}^2(\St_2) +  {\rm Sym}^2(\sigma) \otimes \Lambda^2(\St_2) = \det(\sigma)\St_3+ {\rm Sym}^2(\sigma).$$ 

Since $\det(\lambda)= \det( \sigma \otimes \St_2) = \det(\sigma)^2$, we can take $\det(\lambda)^{1/2}=\det(\sigma)$, and hence, 
\begin{eqnarray}\det(\lambda)^{-1/2}\Lambda^2(\sigma \otimes \St_2) = \St_3+ 
(\det{\sigma})^{-1}{\rm Sym}^2(\sigma) . \end{eqnarray}

Since $\lambda = \sigma \otimes \St_2$ is a conjugate-symplectic representation of $W'_E$, 
$\sigma$  must be a conjugate-orthogonal representation of $W'_E$ which 
by Proposition 6.1 of [GGP2] 
arises (up to a twist by a character of $E^\times$) as basechange 
of a representation of $W'_F$ and the representation $(\det{\sigma})^{-1}{\rm Sym}^2(\sigma)$ extends to a representation of $W'_F$ with values in $\OO(3,\C)$ which by equation $(2)$ must be $\omega_{E/F}\sigma_\pi$.

To conclude --- admitted without all details --- theta lift of representations of $\SL_2(F)$ to $\SO(4,2)(F)$ have parameters which 
are as in the Offen-Sayag theorem, and that conversely, Offen-Sayag parameters come from theta lifts from $\SL_2(F)$.

\vspace{1cm}

\noindent{\bf Acknowledgement:}
There is the work of Mitra-Offen  done simultaneously and independently of us and which appeared in [MO] which also proves that there
are no cuspidal representations   of $\U_{2n}(F)$ distinguished by $\Sp_{2n}(F)$.

The second author held Jean Morlet Chaire at CIRM,  Luminy in 2016 where this work was
 begun. He would like to thank CIRM and its very friendly \'equipe for creating a very
conducive academic environment for work.

The  work of the second author is supported by a  grant of the Government of the Russian
Federation for the state support of scientific research carried out
under the supervision of leading scientists, agreement 14.W03.31.0030 dated 15.02.2018.

\end{document}